\documentclass[10pt,twoside,a4paper]{amsart}

\usepackage[english]{babel}
\usepackage[utf8]{inputenc}

\usepackage{amsfonts,amsthm,amssymb,amsmath,amsthm,latexsym,wasysym,stmaryrd,mathrsfs,dsfont,txfonts}

\usepackage{bigstrut}

\usepackage{pdftricks}
\usepackage{color}
\usepackage[pdftex,all]{xy}

\usepackage{hyperref}
\usepackage{appendix}
\usepackage[justification=centering]{caption}

\graphicspath{{Figures/}}

\theoremstyle{theorem}
\newtheorem {theo}{Theorem}[section]
\newtheorem {theo-intro}{Theorem}[]
\newtheorem*{theo*}{Theorem}
\newtheorem {lemme}[theo]{Lemma}
\newtheorem*{lemme*}{Lemma}
\newtheorem {prop}[theo]{Proposition}
\newtheorem*{prop*}{Proposition}
\newtheorem {cor}[theo]{Corollary}
\newtheorem*{cor*}{Corollary}

\newtheorem*{cor_proof*}{Corollary (of the proof)}

\newtheorem*{conjecture*}{Conjecture}
\theoremstyle{definition}
\newtheorem {defi}[theo]{Definition}
\newtheorem*{defi*}{Definition}
\newtheorem {nota}[theo]{Notation}
\newtheorem*{nota*}{Notation}
\theoremstyle{remark}
\newtheorem {remarque}[theo]{Remark}
\newtheorem*{remarque*}{Remark}

\newtheorem*{warning*}{Warning}

\newtheorem*{remarques*}{Remarks}

\newtheorem*{warnings*}{Warnings}

\newtheorem*{convention*}{Convention}

\newtheorem*{exemple*}{Example}

\newtheorem*{exemples*}{Examples}

\newtheorem*{question*}{Question}

\newtheorem*{questions*}{Questions}

\newtheorem*{fact*}{Fact}
\newtheorem*{acknowledgments}{Acknowledgments}
\newtheorem*{gloss}{Glossary}

\def\N{{\mathds N}}

\def\Z{{\mathds Z}}
\def\e{\varepsilon}
\def\p{\partial}
\def\UnN{\{1,\cdots,n\}}

\def\pcup{\operatornamewithlimits{\cup}\limits}
\def\psqcup{\operatornamewithlimits{\sqcup}\limits}

\def\psum{\operatornamewithlimits{\sum}\limits}

\def\fract#1/#2{\hbox{\leavevmode
  \kern.1em \raise .25ex \hbox{\the\scriptfont0 $#1$}\kern-.1em }\big/
  {\hbox{\kern-.15em \lower .5ex \hbox{\the\scriptfont0 $#2$}} }}

\def\fractt#1/#2{\hbox{\leavevmode
  \kern.1em \raise .25ex \hbox{\the\scriptfont0 $#1$}\kern-.1em
}\lower .2ex\hbox{\Big/}
  {\hbox{\kern-.15em \lower .8ex \hbox{\the\scriptfont0 $#2$}} }}

\def\subfract#1/#2{\hbox{\leavevmode
  \kern.1em \raise .25ex \hbox{\the\scriptfont0 \scriptsize $#1$}\kern-.1em }/
  {\hbox{\kern-.15em \lower .5ex \hbox{\the\scriptfont0 \scriptsize $#2$}} }}

\newcommand{\dessin}[2]{
  \vcenter{\hbox{\includegraphics[height=#1]{#2.pdf}}}}

\newcommand{\func}[3]{
  #1 \colon #2 \longrightarrow #3}
\newcommand{\plong}[3]{
  #1 \colon #2 \hookrightarrow #3}

\def\crashto{\rotatebox{75}{\reflectbox{$\lightning$}}}
\def\crashdown{\raisebox{0cm}{$\rotatebox{270}{$\crashto$}$}}
\def\crashup{\raisebox{0cm}{$\rotatebox{90}{$\crashto$}$}}

\DeclareRobustCommand\refmark[1]{\textsuperscript{\ref{#1}}}

\def\SS{\mathfrak{S}}
 
\def\lk{\textrm{lk}}
\def\vlk{\textrm{vlk}}

\def\AR{\textrm{VC}}
\def\F{\textrm{F}}
\def\CC{\textrm{CC}}
\def\Clasp{\textrm{4-move}}
\def\V{\textrm{V}}
\def\M{\textrm{M}}
\def\RU{\textrm{R1}}
\def\RD{\textrm{R2}}
\def\RT{\textrm{R3}}

\def\OC{\textrm{OC}}
\def\UC{\textrm{UC}}
\def\SV{\textrm{SV}}
\def\SC{\textrm{SC}}
\def\SR{\textrm{SR}}

\def\Sharp{\textrm{BP}}
\def\wSharp{\textrm{wBP}}

\def\P{\textrm{P}}
\def\Pn{\P_n}
\def\wP{\textrm{w}\P}
\def\wPn{\wP_n}

\def\SL{\mathcal{SL}}
\def\SLn{\SL_n}
\def\wSL{\textrm{w}\SL}
\def\wSLn{\wSL_n}

\def\CL{\mathcal{L}}
\def\CLn{\CL_n}
\def\wCL{\textrm{w}\CL}
\def\wCLn{\wCL_n}

\def\SLDn{\textrm{SLD}_n}
\def\vSLD{\textrm{vSLD}}
\def\vSLDn{\textrm{vSLD}_n}

\def\RF{\textrm{RF}}
\def\RFn{{\RF_n}}
\def\Aut{\textrm{Aut}}
\def\AutC{\Aut_{\textrm{C}}}

\def\Cl{\textrm{Cl}}
\def\Op{\textrm{Op}}
\def\Id{\textrm{Id}}
\def\Tube{\textrm{Tube}}

\def\HL{\textrm{HL}}

\def\m{\textrm{mod}}

\def\cprecc{\hbox{$\stackrel{c}{\Rightarrow}$}}
\def\wprecc{\hbox{$\stackrel{w}{\Rightarrow}$}}
\def\vprecc{\rotatebox{300}{$\stackrel{w}{\Rightarrow}$}}
\def\vcong{\rotatebox{90}{$\simeq$}}

\def\psum{\operatornamewithlimits{\sum}\limits}

\begin{document}
\title{Extensions of some classical local moves on knot diagrams}

\author[B. Audoux]{Benjamin Audoux}
         \address{Aix--Marseille Universit\'e, I2M, UMR 7373, 13453 Marseille, France}
         \email{benjamin.audoux@univ-amu.fr}
\author[P. Bellingeri]{Paolo Bellingeri}
         \address{Universit\'e de Caen, LMNO, UMR 6139, 14032 Caen, France}
         \email{paolo.bellingeri@unicaen.fr}
\author[J-B. Meilhan]{Jean-Baptiste Meilhan}
         \address{Universit\'e Grenoble Alpes, IF, UMR 5582, 38000 Grenoble, France}
         \email{jean-baptiste.meilhan@ujf-grenoble.fr}
\author[E. Wagner]{Emmanuel Wagner}
         \address{Universit\'e Bourgogne Franche-Comt\'e, IMB, UMR 5584, 21000 Dijon, France}
         \email{emmanuel.wagner@u-bourgogne.fr}


\date{\today}
\begin{abstract}
In the present paper, we consider local moves on classical and welded 
diagrams: (self-)crossing change, (self-)virtualization, virtual
conjugation, Delta, fused, band-pass and welded band-pass moves. Interrelationship
between these moves is discussed and, for each of
these move, we provide an algebraic classification. 
We address the question of relevant welded extensions for
classical moves in the sense that the classical quotient of classical object
embeds into the welded quotient of welded objects. 
As a by-product, we obtain that all of the above local moves are unknotting operations for welded (long) knots. 
We also mention some topological interpretations for these combinatorial quotients.
\end{abstract}

\maketitle

\section*{Introduction}

Although knot and link theory has its roots and foundations in the topology of
embedded circles in 3-space, its study was early turned
combinatorial by considering generic projections
which can be seen as decorated 4-valent planar graphs. This opened a
new way to think the topology in terms of
combinatorial local
moves. First, ambient isotopies were proved to correspond to Reidemeister moves
\cite{Reidemeister}. Other equivalence relations were then interpreted
as the quotient under some additional local moves: for general homotopy, one
should authorize crossing changes (CC); for link-homotopy, only
self-crossing changes (SC) \cite{Milnor}; and link-homology, introduced by
Murakami and Nakanishi \cite{MN} and Matveev \cite{Matveev}, corresponds to Delta moves ($\Delta$). 
Other local
moves were also investigated, still within some topological
perspectives, such as the band-pass move ($\Sharp$) which is motivated by the
crossing of two bands, but also from more algebraic or even purely
combinatorial considerations. 
These notions straightforwardly
extend to other kinds of knotted objects in dimension 3.
\begin{figure}[h]
    \[
  \begin{array}{ccc}
 \dessin{1cm}{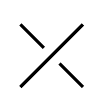}\ \stackrel{\CC}{\longleftrightarrow}
      \dessin{1cm}{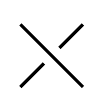}
   &
     \dessin{1cm}{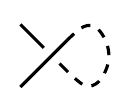}\ \stackrel{\SC}{\longleftrightarrow}
      \dessin{1cm}{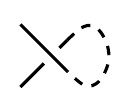}
      \end{array}
  \begin{array}{cc}
  \dessin{1.2cm}{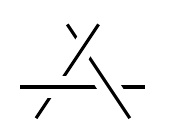} \stackrel{\Delta}{\longleftrightarrow}
        \dessin{1.2cm}{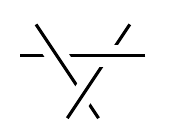}
    &
 \dessin{1.3cm}{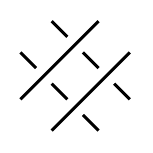}\ \stackrel{\Sharp}{\longleftrightarrow}
      \dessin{1.3cm}{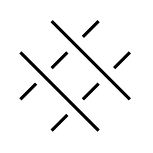}
      \end{array}
  \]
  \caption{Classical local moves}
  \label{fig:ClassLocMoves0}
\end{figure}

Forgetting the planarity assumption for the decorated 4-valent graphs gives rise to
the notion of \emph{virtual links}, introduced from the diagram point of view
by Kauffman \cite{Kauffman} and from the Gauss diagram point of view
by Goussarov, Polyak and Viro \cite{GPV}. In this virtual context, two 
forbidden local moves emerged: the over- and the under-commute
moves. The \emph{welded} theory, first introduced in the braid context
by Fenn, Rim{\'a}nyi and Rourke \cite{FRR}, is defined by allowing one of them;
the \emph{fused} theory, already mentionned by Kauffman in
\cite{Kaufmann},  by allowing both. 

Welded knots and links provide a sensible extension of usual knot theory in the sense that two classical links are equivalent as welded objects if and only if they are classically equivalent. 
In 2000, Satoh provided another topological
motivation for welded knotted objects by generalizing a construction ---given fourty years
earlier, in the classical case, by Yajima \cite{Yajima}--- that
inflates diagrams into embedded tori in 4-space which bound
immersed solid tori with only ribbon singularities. The resulting map,
so-called $\Tube$ map, is surjective but its injectivity remains an intriguing
question: 
false for welded links \cite{Winter,Jess}, true for welded braids \cite{BH} and undetermined for welded string links.

In \cite{vaskho}, the authors used the $\Tube$ map to classify
ribbon tubes and ribbon 2-torus links --- which are a 2--dimensional analogue of
string links and links --- up to link-homotopy. Along the paper, several
phenomena emerged:
\begin{enumerate}
\item\label{item:1} link-homotopy among ribbon objects is
  generated by the image through the $\Tube$ map of a single local
  move, namely the self-virtualization (SV), and up to this move, the
  $\Tube$ map is one-to-one;
\item\label{item:2} as in the classical case, every welded
  string link is link-homotopic to a welded braid, that is,  the map from welded pure braids to
  welded string links up to self-virtualization is surjective;
\item \label{item:3} the given classification of welded string links up to self-virtualization  is a natural extension of the classification of
  classical string links up to link-homotopy given by Habegger and Lin
  \cite{HL}. As such, it suggests that self-virtualization is a
  natural welded extension of the classical self-crossing change, in
  the sense that the embedding of planar 4-valent graphs into
  general 4-valent graphs induces an embedding of classical
  string links up to self-crossing change into welded string links up to
  self-virtualization.
\end{enumerate}

Point (\ref{item:2}) has been developed in \cite{vaskho2}.
Point (\ref{item:1}) is raised at the end of the present introduction, but the paper essentially pushes further the analysis of point (\ref{item:3}) by discussing 
the welded extensions of the classical 
$\Delta$ and $\Sharp$ moves. 
In doing so, we define several candidates for such extensions and compare
them, carrying on a work initiated in the classical case by
Murakami--Nakanishi \cite{MN} and Aida
\cite{Aida}. An unexpected outcome is that a given classical local move may
admit several distinct welded extensions (see {\it e.g.} Proposition \ref{prop:FextendsD}). 
Specifically, we consider the following non classical local moves (see section \ref{sec:LocalMoves} for details): 
\begin{figure}[h]
\[  \begin{array}{ccc}
\dessin{1.125cm}{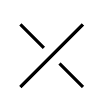}\ \stackrel{\V}{\longleftrightarrow}
      \dessin{1.125cm}{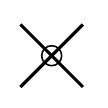}
    &
    \dessin{1.125cm}{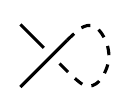}\ \stackrel{\SV}{\longleftrightarrow}
      \dessin{1.125cm}{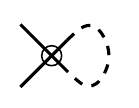}
      &
    \dessin{.9375cm}{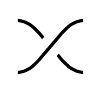} \stackrel{\AR}{\longleftrightarrow}
        \dessin{.9375cm}{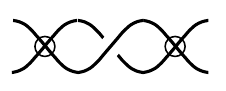}
       \end{array}        \]
\[  \begin{array}{ccc}
\dessin{1.5cm}{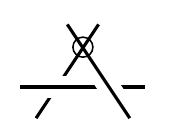} \stackrel{\F}{\longleftrightarrow}
        \dessin{1.5cm}{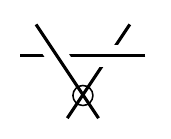}
&
\dessin{1.5cm}{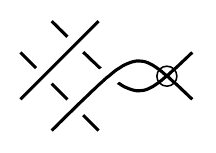} \stackrel{\ \wSharp}{\longleftrightarrow}
        \dessin{1.5cm}{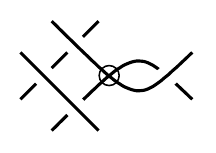}
 \end{array}        \]
  \caption{Non classical local moves}
  \label{fig:WeldedLocMoves0}
\end{figure}

We provide in Theorems \ref{th:LocMovesOrder} and \ref{th:Del=>SC} an ordering between the above classical and non classical local moves. Notice that these results hold for all types of welded knotted objects.  
Moreover, in the case of links and string links, we provide a complete classification under these moves, as stated below. 

Recall that links up to $\Delta$ moves are classified by the linking numbers \cite[Thm 1.1]{MN}, while 
links up to $\Sharp$ are classified by the modulo 2 reduction of  $\psum_{1\leq k\neq i\leq n}\lk_{ik}$
\cite{MN,Kawauchi}. 
The main results of this paper can be summarized as follows. 
\begin{theo-intro}\label{th:ColoredLinksClassification}
  \begin{itemize}
  \item[]
  \item Welded links up to $\F$ are classified by the virtual linking
    numbers.
  \item Welded links up to $\AR$ are classified by the
    $(\vlk_{ij}+\vlk_{ji})$'s.
  \item Welded links up to $\CC$ are classified by the
    $(\vlk_{ij}-\vlk_{ji})$'s.
  \item Welded links up to $\wSharp$ are classified by the
    $\vlk_{i*}^\m$'s and the $(\vlk_{ij}^\m+\vlk_{ji}^\m)$'s.
  \end{itemize}
\end{theo-intro}
Here, the \emph{virtual linking number} $\vlk_{ij}$ is the welded 
link invariant which counts, with signs, the crossings where the $i^\textrm{th}$ component overpasses the $j^\textrm{th}$ component,  $\vlk_{ij}^\m$ denotes its modulo 2 reduction, and $\vlk_{i*}^\m$ denotes the modulo 2 reduction of  $\psum_{1\leq k\neq i\leq n}\vlk_{ik}$.  

As a consequence, we obtain that $\AR$, $\Delta$, $\F$, $\Sharp$ and $\wSharp$ are all unknotting operations for welded knots, which recovers and extends a result recently proved by S. Satoh \cite{Satoh2} using a different approach. We actually show the stronger result that these are all unknotting operations for welded long knots. 
Another consequence is the following extension result. 
\begin{theo-intro}\label{th:ColoredLinksExtension}
  \begin{itemize}
  \item[]
  \item Links up to $\Delta$ embed into welded links up to $\F$.
  \item Links up to $\AR$ embed into welded links up to $\F$.
  \item Links up to $\Sharp$ embed into welded links up to $\wSharp$.
  \end{itemize}
\end{theo-intro}
Note that the classification of welded links up to $\F$ has been independently proved in \cite{Nas} with a completely different and algebraic approach. This completes a previous result of Fish and Keyman \cite[Thm 2]{FK2} (see also \cite[Thm 3.7]{BBD} for a shorter proof) stating that  
fused links with only classical crossings are classified by linking numbers.

To prove these results, we provide algebraic classifications of all the considered local moves for string links. 
For each of them, we give an explicit group isomorphism between the
quotient space of (welded) string links and a power of $\Z$ or $\Z_2$.  
Our main tool will be the theory of Gauss diagrams, mentioned earlier, which is an even more combinatorial
alternative to describe virtual diagrams. If the virtual diagrams are a pleasant
tool to picture local operations, Gauss diagrams appear to be more
efficient to handle global manipulations. In the present paper, we shall adopt and use both points of view in parallel.
\medskip 

Let us conclude this introduction with a few comments returning back to topology.

In \cite{MN}, Murakami and Nakanishi introduced a notion of
\emph{link-homology} which can be rephrased as the quotient where two
elements are identified whenever the image of each strand in the homology groups 
of the complement of the other strands are the
same. They noted that the classification of links up to $\Delta$ by the
linking numbers implies that the link-homology is generated by the
$\Delta$ move or, equivalently, that (string)  links up to $\Delta$ describe  (string) links up to link-homology. 
Similarly, string links up to self-crossing changes were studied in \cite{HL} as the group of string
links up to link-homotopy --- that is the topological quotient where
each connected component is allowed to cross itself.

As already mentioned, welded (string) links also have a 
topological interpretation, via Satoh's Tube map \cite{Satoh}. This topological interpretation is however partial since the $\Tube$ map is surjective but not injective. Indeed,
performing $\SR$ ---a local move depicted in Figure \ref{fig:WeldedLocMoves}--- on each classical crossing and reversing the
orientation on a given link diagram produces another diagram with same
image through $\Tube$; see \cite[Thm 3.3]{Winter} or \cite[Prop 2.7]{Jess}. It is
moreover still unknown whether this move generates all the kernel of
$\Tube$ for welded links.

In \cite{vaskho}, the authors applied the $\Tube$ map to welded string
links, producing \emph{ribbon tubes}. 
It is shown in \cite[Prop 3.16]{vaskho} that $\SV$
generates ribbon link-homotopy on ribbon tubes, and that the $\Tube$
map is injective on the quotient, thus producing a full topological
interpretation for welded string links up to $\SV$. Furthermore, the virtual linking number $\vlk_{ij}$ corresponds to the evaluation of any longitude ---that is any path from one
boundary component to the other--- of the $j^\textrm{th}$ tube in $\Z$ seen as the first homology
group of the complement of the $i^\textrm{th}$ tube. 
Since longitudes are the only subspaces of a ribbon tube component that may have a nontrivial image 
in the homology groups of the complement of the other components, 
it follows from Proposition
\ref{prop:ClassificationF} that welded string links up to $\F$ describe faithfully ribbon
tubes up to (the natural extension of
  Murakami and Nakanishi's notion of) link-homology.

Let us mention here that the term ``link-homology'' is also used in the litterature as a synonym for bordance, which is a weakening of the concordance obtained by allowing any cobordism. It appears within the framework of usual knot theory (see  \cite{Sanderson1, Sanderson2, CKSS} for a classification result), but also in the context of virtual knot theory seen as links in thickened surfaces up to isotopy and (de)stabilization. Carter, Kamada and Saito proved in \cite{CKS} that virtual links up to virtual link-homology are classified by virtual linking numbers. As a byproduct, we obtain that the two forbidden moves generate virtual link-homology.

\medskip 

The paper is organized as follows. In Section \ref{sec:definitions},
the central objects of the paper ---classical/welded
string links and the notion of local move--- are defined from both the virtual and
the Gauss diagrams points of view.
Section \ref{sec:LocalMoves} introduces all the considered local moves
(see Figures \ref{fig:ClassLocMoves} and \ref{fig:WeldedLocMoves})
and study their interrelationship. 
In Section \ref{sec:Classification}, virtual linking numbers are 
used to provide a complete algebraic
classification of welded string links up to each considered local
move. This algebraic identifications are then used to discuss welded
extensions for classical (self)-crossing changes, $\Delta$ and
$\Sharp$ moves. Section \ref{sec:Others} builds on the string link case to
address other kinds of knotted objects such as  links and
pure braids. 

\begin{gloss}
Throughout the paper, the various local moves studied in this paper will be denoted by the notation introduced in their defining figures.  
For the reader's convenience, we list below these various acronyms, their meaning and the references to their definition. 
\[
\begin{array}{cccccc}
 Ri\, ;\,\,i=1,2,3 & &  \textrm{Reidemeister move i}  & & \textrm{Figure \ref{fig:AllReidMoves}} \\
 vRi\, ;\,\,i=1,2,3 & &  \textrm{virtual Reidemeister move i}  & & \textrm{Figure \ref{fig:AllReidMoves}} \\
 \OC & &  \textrm{Over-commute move}  & & \textrm{Figure \ref{fig:AllReidMoves}} \\
 \UC & &  \textrm{Under-commute move}  & & \textrm{Proposition \ref{prop:F<=>UC}} \\
 \CC & & \textrm{Crossing Change} & & \textrm{Figures \ref{fig:ClassLocMoves0} and \ref{fig:ClassLocMoves}}\\
 \SC & & \textrm{Self-crossing Change} & & \textrm{Figures \ref{fig:ClassLocMoves0} and \ref{fig:ClassLocMoves}}\\
 \Delta & &  \textrm{Delta move}  & & \textrm{Figures \ref{fig:ClassLocMoves0} and \ref{fig:ClassLocMoves}} \\
 \Sharp & &  \textrm{unoriented band-pass move}  & & \textrm{Figures \ref{fig:ClassLocMoves0} and \ref{fig:ClassLocMoves}} \\
 \V  & &  \textrm{Virtualization move}  & & \textrm{Figures \ref{fig:WeldedLocMoves0} and \ref{fig:WeldedLocMoves}} \\
 \SV & &  \textrm{Self-virtualization move}  & & \textrm{Figures \ref{fig:WeldedLocMoves0} and \ref{fig:WeldedLocMoves}} \\
 \AR & &  \textrm{Virtual conjugation move}  & & \textrm{Figures \ref{fig:WeldedLocMoves0} and \ref{fig:WeldedLocMoves}} \\
 \SR & &  \textrm{Sign reversal move}  & & \textrm{Figure \ref{fig:WeldedLocMoves}} \\
 \F  & &  \textrm{Fused move}  & & \textrm{Figures \ref{fig:WeldedLocMoves0} and \ref{fig:WeldedLocMoves}} \\
 \wSharp & &  \textrm{welded band-pass move}  & & \textrm{Figures \ref{fig:WeldedLocMoves0} and \ref{fig:WeldedLocMoves}}      
\end{array}
\]
We also  note here, as a point of convention, that the same acronym shall often be used when refering to the equivalence relation on diagrams generated by the corresponding local move. 
\end{gloss}

\begin{acknowledgments}
This work began during the summer school \emph{Mapping class groups, $3$- and $4$-manifolds} in Cluj in July 2015; 
the authors thank the organizers for the great working environment. They also thank the referee for the careful reading and useful suggestions.
The research of the authors was partially supported by  the late French ANR research project ``VasKho'' ANR-11-JS01-002-01. 
\end{acknowledgments}



\section{Classical and welded string links} \label{sec:definitions}



We first introduce, in two different but equivalent ways, the main objects of this paper. All along the text, $n$ shall be a   positive integer.

\subsection{Virtual diagrams}\label{subsec:Diagrams}

Fix $n$  real numbers $0<p_1<\cdots<p_n<1$.  
\begin{defi}\label{def:diagrStringLink}
A \emph{virtual string link diagram} is an immersion of $n$ oriented intervals $\psqcup_{i\in\UnN}I_i$ in $I\times I$, called \emph{strands}, such that
\begin{itemize}
\item for each $i\in\UnN$, the strand $I_i$ has boundary $\p I_i=\{p_i\}\times\{0,1\}$ and is oriented from $\{p_i\}\times\{0\}$ to $\{p_i\}\times\{1\}$;
\item the singular set is a finite number of transverse double points;
\item each double point is labelled, either as a \emph{positive crossing}, as a \emph{negative crossing}, or as a \emph{virtual crossing}. Positive and negative crossings are also called \emph{classical crossings}. 
\end{itemize}
Strands are naturally ordered by the order of
their endpoints on either $I\times\{0\}$ or $I\times\{1\}$.\\
A virtual string link diagram which has no virtual crossing is said to be \emph{classical}.\\
Up to ambient isotopy and reparametrization,  the set of virtual string link diagrams is naturally endowed with a structure of monoid by the stacking product, 
where the unit element is the trivial diagram $\pcup_{i\in \UnN} \{p_i\}\times I$; we denote this monoid by $\vSLDn$ and its submonoid made of classical string link diagrams by $\SLDn$. We denote by $\plong{\iota}{\SLDn}{\vSLDn}$ the natural injection.\\
Crossings where the two preimages belong to the same strand are called \emph{self-crossings}.   
\end{defi}

We shall use the usual drawing convention for crossings: 
  \[
  \begin{array}{ccccc}
    \dessin{1.2cm}{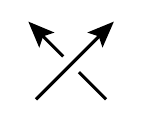}&\hspace{1cm}&\dessin{1.2cm}{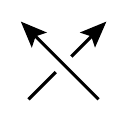}&\hspace{1cm}&\dessin{1.2cm}{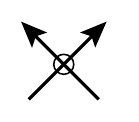}\\
    \textrm{positive crossing}&&\textrm{negative crossing}&&\textrm{virtual crossing}
  \end{array}.
\]
\medskip

\begin{defi}\label{def:localmove}
  A \emph{local move} is a transformation that changes a diagram only
  inside a disk. It is specified by the contents of the disk, before
  and after the move. In our context, the contents shall be pieces of
  strands, without any specified orientation, which may classically and virtually cross
  themselves. If the disk does not contain
  any virtual crossing neither before nor after the move, we say the
  local move is \emph{classical}.\\
To represent a local move, we shall draw only the disk where the move
occurs. Examples are given in Figures \ref{fig:ClassLocMoves0} and \ref{fig:WeldedLocMoves0}.
\end{defi}

\begin{defi}\label{def:wStringLink}
A \emph{string link} is an equivalence class of $\SLDn$ under the three classical Reidemeister moves.
We denote by $\SLn$ the set of string links; it is a monoid
with composition induced by the stacking product.\\
  A \emph{welded string link} is an equivalence class of $\vSLDn$ under the 
welded Reidemeister moves, which are the classical and the virtual Reidemeister moves,  
together with the mixed and the over-commute ($\OC$) moves, given in Figure \ref{fig:AllReidMoves}. 
We denote by $\wSLn$ the set of welded string links; it is a monoid
with composition induced by the stacking product.\\
Elements of $\SL_1$ and $\wSL_1$ are also called, respectively, long knots and welded long knots.
\end{defi}

\begin{figure}
\[
  \begin{array}{c}
    \begin{array}{ccc}
      \dessin{1.5cm}{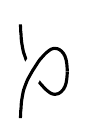} \stackrel{\textrm{R1}}{\longleftrightarrow}  \dessin{1.5cm}{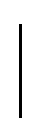} \leftrightarrow \dessin{1.5cm}{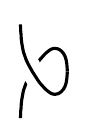} \, &
      \,   \dessin{1.5cm}{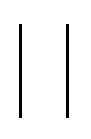} \stackrel{\textrm{R2}}{\longleftrightarrow}  \dessin{1.5cm}{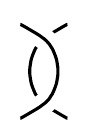} \, 
      & \, 
      \begin{array}{c}
\dessin{1.5cm}{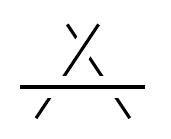} \stackrel{\textrm{R3}}{\longleftrightarrow}  \dessin{1.5cm}{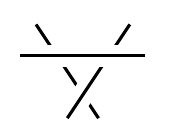}\\
\dessin{1.5cm}{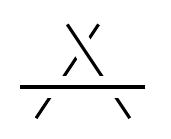} \stackrel{\textrm{R3}}{\longleftrightarrow}  \dessin{1.5cm}{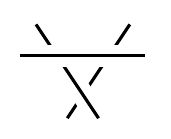}
      \end{array}
      \end{array}
     \\
    \textrm{classical Reidemeister moves}\\[0.4cm]
    \begin{array}{ccc}
   \dessin{1.5cm}{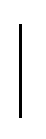}
    \stackrel{\textrm{vR1}}{\longleftrightarrow}  \dessin{1.5cm}{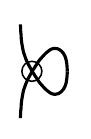} \, & \, 
    \dessin{1.5cm}{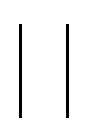} \stackrel{\textrm{vR2}}{\longleftrightarrow}  \dessin{1.5cm}{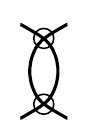}
   \, & \, 
   \dessin{1.5cm}{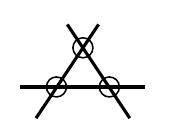} \stackrel{\textrm{vR3}}{\longleftrightarrow}  \dessin{1.5cm}{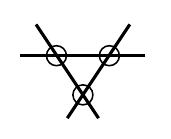}
   \end{array}
   \\
    \textrm{virtual Reidemeister moves}\\[0.4cm]
    \begin{array}{cc}
      \dessin{1.5cm}{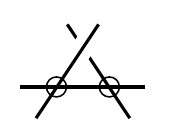}\stackrel{\textrm{Mixed}}{\longleftrightarrow} \dessin{1.5cm}{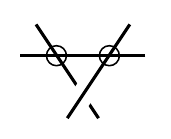} \, & \, 
      \dessin{1.5cm}{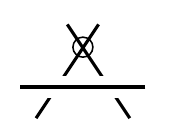}\ \stackrel{\OC}{\longleftrightarrow}
      \dessin{1.5cm}{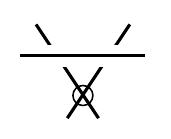}
    \end{array}
    \\
    \textrm{Mixed and OC moves}
  \end{array}
\]
  \caption{Welded Reidemeister move}\label{fig:AllReidMoves}
\end{figure}
%

String-links can be seen as an intermediate object between braids and links. For convenience and since we 
shall mention them, 
we give short definitions of these objects:

\begin{itemize}
\item compared with string links and welded string links, \emph{pure
    braids} and \emph{welded pure braids} are defined by requesting,
  in addition, that the immersed intervals are monotone with respect to
  the second coordinate. Ambient isotopies are then also requested to respect this monotony. The stacking product induces then group
  structures that we denote, respectively, by $\Pn$ and $\wPn$;
\item \emph{links} and \emph{welded links} are defined by replacing,
  in Definition \ref{def:diagrStringLink}, the disjoint union of oriented
  intervals $\psqcup_{i\in\UnN}I_i$ by a disjoint union of oriented
  circles $\psqcup_{i\in\UnN}S^1_i$ (ignoring the points $p_i$). 
  Note that (welded) links are thus implicitly equipped with an enumeration of the set of its connected components.  
  We denote by  $\CLn$ and $\wCLn$ the sets of  links and welded links.  
  Elements of $\CL_1$ and $\wCL_1$ are also called, respectively, knots and welded knots.
\end{itemize}


\subsection{Gauss diagrams}

Welded string links 
can by definition be represented by virtual string link diagrams which can, in turn, 
be alternatively described in terms of \emph{Gauss diagrams}.

\begin{defi} \label{def:GD}
A \emph{Gauss diagram} is defined over $n$ ordered and
  oriented intervals, called \emph{strands}, as a finite set of triplets $(t,h,\sigma)$,
called \emph{arrows},
where $t$ and $h$, called respectively the \emph{tail} and the \emph{head} of
the arrow, are elements of the strands and $\sigma\in\{\pm1\}$ is a sign.
Tails and heads, also called \emph{endpoints} or \emph{ends}, are all distinct and considered up to
orientation-preserving homeomorphisms of the strands.\\
The strands are represented by parallel upward thick intervals arranged in
increasing order, and each arrow by
an actual thin arrow, going from its tail to its head, labelled by its sign.\\
Arrows having both ends on the same strand are called \emph{self-arrows}.
\end{defi}
See the right-hand side of Figure \ref{fig:cestcadeaucamfaisplaisir} for an example. 
\begin{figure}[h!]
\[
\begin{array}{ccc}
\dessin{3cm}{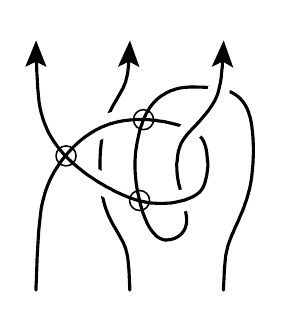} 
&\hspace{.5cm}\leftrightsquigarrow\hspace{.5cm}&
\dessin{3cm}{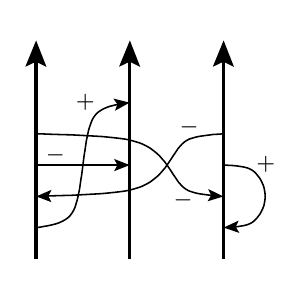} 
   \end{array}
 \] 
 \caption{A virtual diagram, and the corresponding Gauss diagram}\label{fig:cestcadeaucamfaisplaisir}
\end{figure}

\begin{defi}\label{def:localmoveGD}
  A \emph{local move} is a transformation that changes arrows only on
  a given finite union of portions of strands. It is specified by the
  portions with their arrows, before
  and after the move, and it is assumed that no other arrow has an
  endpoint on these portions.
\end{defi}

To represent a local move, we shall draw only the portions where the
move occurs by parallel thick intervals, without any specified
orientation nor ordering.
The reader should be aware that, when the local move is applied to a given Gauss diagram,
the portions should be reordered ---and even possibly put one above others if they are
part of a same strand--- and possibly reversed to get upward; 
(s)he should also keep in mind that there is a non represented part,
which is identical on each side of the move, but that no arrow can
connect the non represented part to the represented one. Examples are given in Figure \ref{fig:ReidMoves} and on the right hand-side of Figures \ref{fig:ClassLocMoves} and \ref{fig:WeldedLocMoves}.
\medskip 

There is a one-to-one correspondence between Gauss diagrams up to ambient isotopy and
virtual diagrams up to virtual Reidemeister and mixed moves.
It associates a Gauss diagram to any virtual
diagram so that the set of positive and negative
crossings in the virtual diagram are in one-to-one correspondence
with, respectively, the set
of $+1$--labelled and $-1$--labelled arrows in the Gauss diagram. This procedure is, for example, described in
\cite[Section 4.5]{vaskho}, and is illustrated in Figure \ref{fig:cestcadeaucamfaisplaisir}.  
Local moves on virtual diagrams have Gauss diagrams counterparts.  
Figure \ref{fig:ReidMoves} gives the Gauss diagram versions of the classical Reidemester moves. 
Note that the Gauss diagram counterparts of the virtual Reidemeister and mixed moves are actually trivial since they do not affect 
any classical crossing. 
Throughout the paper, we shall use indifferently one or the other description. 
\begin{figure}[h]
\[
\begin{array}{cccc}
  \dessin{1.8cm}{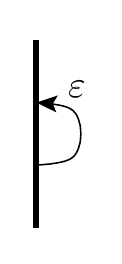} \stackrel{\RU}{\leftrightarrow}
  \dessin{1.8cm}{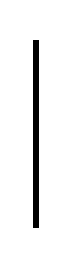}
  \, & \, 
  \dessin{1.8cm}{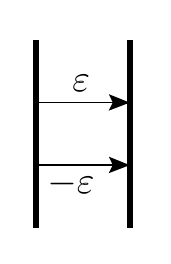} \stackrel{\RD}{\leftrightarrow}
  \dessin{1.8cm}{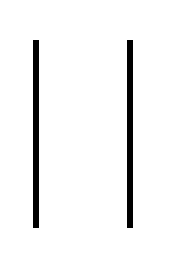}
  \, & \, 
  \dessin{1.8cm}{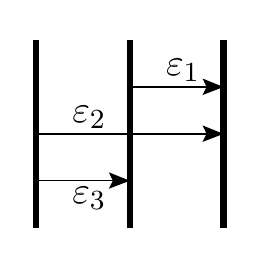} \stackrel{\RT}{\leftrightarrow} \dessin{1.8cm}{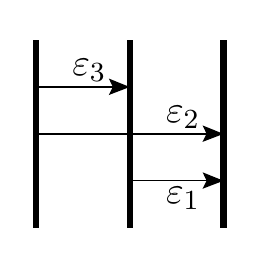}
  \, & \,  
  \dessin{1.8cm}{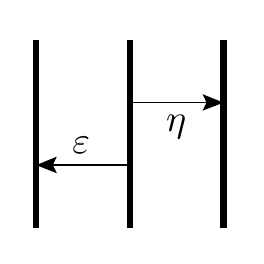}\  \stackrel{\OC}{\leftrightarrow}\ \dessin{1.8cm}{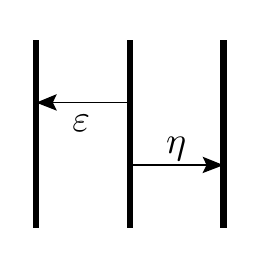}
\end{array}
\] 
  \caption{Welded Reidemeister moves on Gauss diagrams\\
  {\footnotesize There is a sign condition for applying move R3, namely that  $\e_i\delta_i=\e_j\delta_j$, where $\delta_k=1$ if the the $k^\textrm{th}$ strand, read from left to right, is oriented upward and $-1$ otherwise}}\label{fig:ReidMoves}
\end{figure}

\noindent This correspondence yields a faithful representation of welded string links, by Gauss diagrams up to the welded Reidemeister moves depicted in Figure \ref{fig:ReidMoves}. 

\medskip 

Welded pure braids and welded links also enjoy Gauss
diagram descriptions:
\begin{itemize}
\item welded pure braids are faithfully represented, up to the welded Reidemeister moves, by Gauss diagrams
  with only horizontal arrows. This result, usually considered as folklore (see for instance \cite{WKO1}), is a consequence of a
similar result on virtual braids \cite[Prop 2.24]{Ci}  (see \cite{DaTh} for a
complete proof in the welded case);
\item by replacing intervals by oriented circles in Definition \ref{def:GD}, we obtain a tool that, up to welded Reidemeister moves, faithfully represents welded links. Gauss diagrams were actually first defined over a single circle to describe knots; see for instance \cite{GPV} or  \cite{F}.
\end{itemize}


\section{Local moves and their relations}
\label{sec:LocalMoves}


In this section, we introduce several local moves and study their interrelationship.

\subsection{Local moves}

In Figures \ref{fig:ClassLocMoves} and \ref{fig:WeldedLocMoves}, we introduce the different local moves that we shall study in detail.

We first consider the classical local moves which were already presented in Figure \ref{fig:ClassLocMoves0}. 
\begin{figure}[h]
    \[
  \begin{array}{ccc}
    \dessin{1.125cm}{CC_1} \stackrel{\CC}{\longleftrightarrow}
      \dessin{1.125cm}{CC_2}
    &\hspace{.5cm}\leftrightsquigarrow\hspace{.5cm}&
                     \dessin{1.8cm}{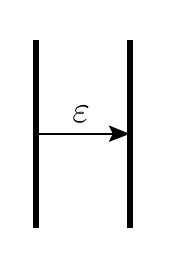} \stackrel{\CC}{\longleftrightarrow}
      \dessin{1.8cm}{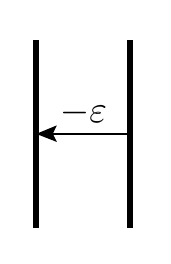}\\

   \dessin{1.125cm}{SC_1} \stackrel{\SC}{\longleftrightarrow}
      \dessin{1.125cm}{SC_2}
    &\hspace{.5cm}\leftrightsquigarrow\hspace{.5cm}&
                     \dessin{1.8cm}{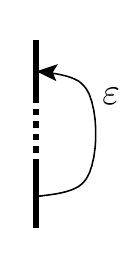} \stackrel{\SC}{\longleftrightarrow}
      \dessin{1.8cm}{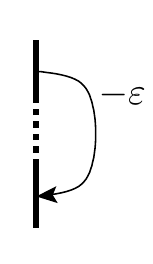}\\

\dessin{1.5cm}{Del_1} \stackrel{\Delta}{\longleftrightarrow}
        \dessin{1.5cm}{Del_2}
    &\hspace{.5cm}\leftrightsquigarrow\hspace{.5cm}&
                     \dessin{1.8cm}{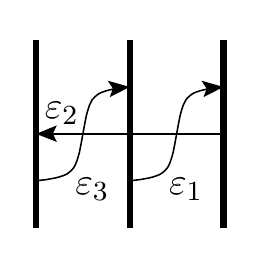} \stackrel{\Delta}{\longleftrightarrow}
      \dessin{1.8cm}{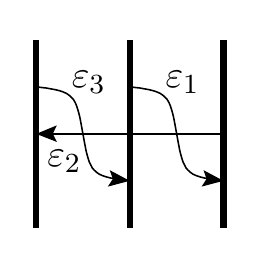}\\

    \dessin{1.69cm}{Sharp_1} \stackrel{\Sharp}{\longleftrightarrow}
      \dessin{1.69cm}{Sharp_2}
    &\hspace{.5cm}\leftrightsquigarrow\hspace{.5cm}&
                     \dessin{1.8cm}{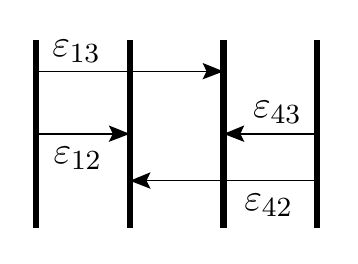} \stackrel{\Sharp}{\longleftrightarrow}
      \dessin{1.8cm}{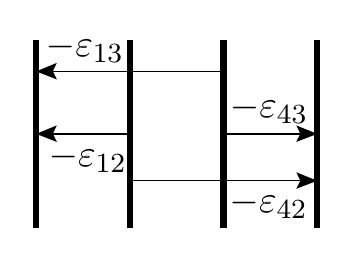}
  \end{array}
  \]
  \caption{Classical local moves\\
    {\footnotesize The Gauss diagram version of move $\Delta$, resp. of $\Sharp$, is subject to the condition that $\e_i\delta_i=\e_j\delta_j$, resp. that $\e_{ij}\e_{kl}=\delta_i \delta_j\delta_k\delta_l$, where $\delta_k=1$ if the the $k^\textrm{th}$ strand, read from left to right, is oriented upward and $-1$ otherwise}}
  \label{fig:ClassLocMoves}
\end{figure}


The crossing change $\CC$ is certainly the simplest and most natural local move in classical knot theory. Its refinement $\SC$ requires the additional self-connectedness 
condition that the two involved pieces of strand belong to the same strand. Note that, although the modification remains local, checking that the pieces are connected is not. This latter move was introduced by Milnor in \cite{Milnor} as a generating move for link-homotopy and furthermore studied by Habegger and Lin in \cite{HL}.

The $\Delta$ move was introduced by Murakami and Nakanishi in \cite{MN} and by Matveev in \cite{Matveev} as a local and combinatorial incarnation of the link-homology quotient of links. There exists another representation of this move, which is given by its mirror image, which is easily checked to be equivalent,  see \cite[Figure 1.1(c)]{MN}.
Similar observations hold for each of the local moves introduced below, and we shall only give one formulation and freely use equivalent versions, leaving as an exercice to the reader to check that they are indeed equivalent. 

The $\Sharp$ move is the unoriented counterpart of the band-pass move, introduced by Murakami in \cite{Murakami} as an alternative unknotting operation for knots.

\medskip

Let us now turn to non classical local moves.
\begin{figure}[h]
    \[
  \begin{array}{ccc}
   \dessin{1.125cm}{V_1} \stackrel{\V}{\longleftrightarrow}
      \dessin{1.125cm}{V_2}
    &\hspace{.5cm}\leftrightsquigarrow\hspace{.5cm}&
                     \dessin{1.8cm}{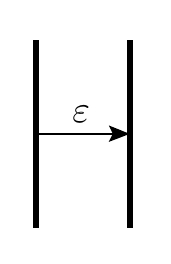} \stackrel{\V}{\longleftrightarrow}
      \dessin{1.8cm}{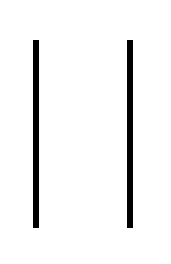}\\

\dessin{1.125cm}{SV_1} \stackrel{\SV}{\longleftrightarrow}
      \dessin{1.125cm}{SV_2}
    &\hspace{.5cm}\leftrightsquigarrow\hspace{.5cm}&
                     \dessin{1.8cm}{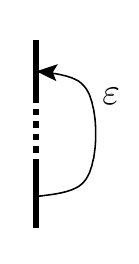} \stackrel{\SV}{\longleftrightarrow}
      \dessin{1.8cm}{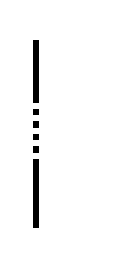}\\

\dessin{.9375cm}{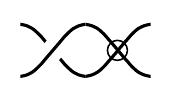} \stackrel{\AR}{\longleftrightarrow}
        \dessin{.9375cm}{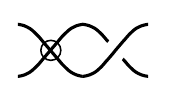}
    &\hspace{.5cm}\leftrightsquigarrow\hspace{.5cm}&
                     \dessin{1.8cm}{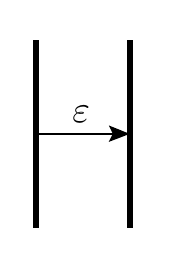} \stackrel{\AR}{\longleftrightarrow}
      \dessin{1.8cm}{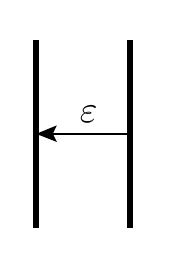}\\

   \dessin{.9375cm}{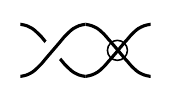} \stackrel{\SR}{\longleftrightarrow}
      \dessin{.9375cm}{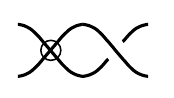}
    &\hspace{.5cm}\leftrightsquigarrow\hspace{.5cm}&
                     \dessin{1.8cm}{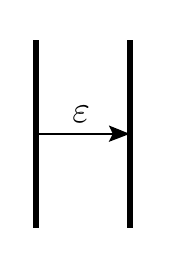} \stackrel{\SR}{\longleftrightarrow}
      \dessin{1.8cm}{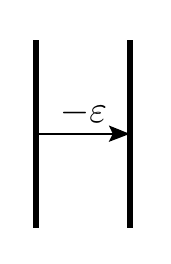}\\

\dessin{1.5cm}{F_3} \stackrel{\F}{\longleftrightarrow}
        \dessin{1.5cm}{F_4}
    &\hspace{.5cm}\leftrightsquigarrow\hspace{.5cm}&
                     \dessin{1.8cm}{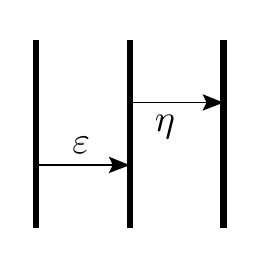} \stackrel{\F}{\longleftrightarrow}
      \dessin{1.8cm}{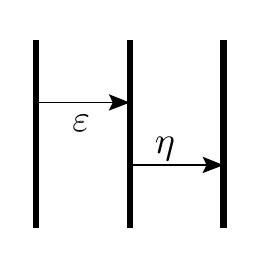}\\

\dessin{1.5cm}{wSharp_1} \stackrel{\ \wSharp}{\longleftrightarrow}
        \dessin{1.5cm}{wSharp_2}
    &\hspace{.5cm}\leftrightsquigarrow\hspace{.5cm}&
                     \dessin{1.8cm}{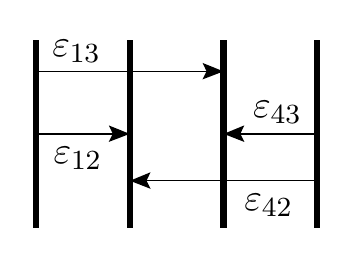} \stackrel{\ \wSharp}{\longleftrightarrow}
      \dessin{1.8cm}{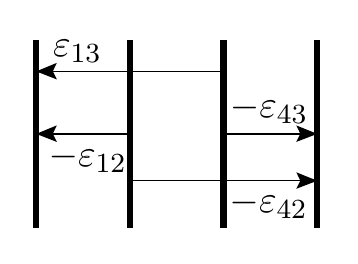}
  \end{array}
  \]
  \caption{Non classical local moves}
  \label{fig:WeldedLocMoves}
\end{figure}

The \emph{virtualization} move $\V$ ---and its self-connected refinement $\SV$---, simply replaces a classical crossing, resp. self-crossing, by a virtual one, or vice-versa. 
It is fairly obvious that $\V$ is an unknotting operation for welded knotted objects. 

The \emph{virtual conjugation} move $\AR$ is best known in the literature ---where it is usually refered to as the virtualization move--- under the form given in Figure \ref{fig:WeldedLocMoves0}. In this paper, it shall be convenient to use the equivalent reformulation given in Figure \ref{fig:WeldedLocMoves}. 

The \emph{sign reversal} move $\SR$ is a composition of the $\AR$ and $\CC$ moves. 

From the Gauss diagram point of view, the moves $\V$, $\SV$, $\AR$ and $\SR$ are the  simplest and most natural local moves, since they all involve a single arrow which is modified by, respectively, being removed/added, having its orientation reversed or having its sign reversed. 

Note that the over-commute move $\OC$ of Figure \ref{fig:AllReidMoves} can be interpreted as allowing adjacent tails to cross one each other. Similarly, move $\F$ can be seen as  a  commutation between a tail and an adjacent head.  
Furthermore, as we shall see later, this move is equivalent, up to $\OC$, to the under-commute move $\UC$, so that it actually allows any pair of adjacent endpoints to commute. In other words, $\F$ defines the \emph{fused} quotient of welded objects, introduced by Kauffman and Lambropoulou in \cite{KL1,KaL}.

The move $\wSharp$ can be seen as a welded analogue of the
classical $\Sharp$ move. 
Note that its Gauss diagram incarnation
should require some sign restrictions, but as we shall prove in a
diagrammatical way that $\SR$ can be realized using $\wSharp$, they
can be released.
\begin{remarque}
 The $\wSharp$ move may appear asymmetric ---and thus unnatural--- to the reader, but turns out to be the simplest candidate for a welded analogue
of the $\Sharp$ move, in the sense of Theorem \ref{th:LocMovesOrder} and Proposition \ref{prop:wsharp_sharp}. 
As a matter of fact, one can consider the following more symmetric version:
\[
\dessin{1.5cm}{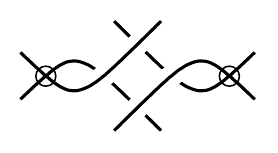} \longleftrightarrow
        \dessin{1.5cm}{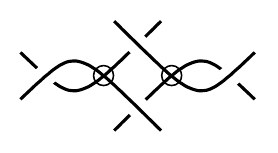}
    \hspace{.5cm}\leftrightsquigarrow\hspace{.5cm}
                     \dessin{1.8cm}{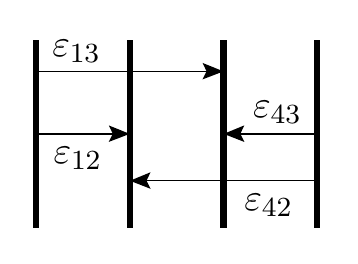} \longleftrightarrow
      \dessin{1.8cm}{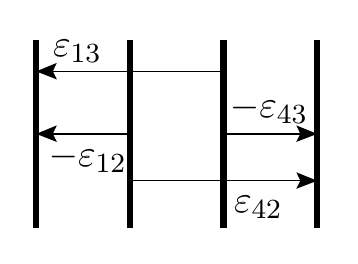}
\] 
and show that it is actually equivalent to $\wSharp$; in practice, however, such a symmetrized version of $\wSharp$ is less convenient, 
since it involves more crossings. 
\end{remarque}

\medskip

To conclude this section, we introduce some generic notation.

\begin{nota}
  For any local move $\mu$, we denote by $\wSLn^\mu$ the quotient of $\wSLn$ under the move $\mu$.\\
  If $\mu$ is classical, we furthermore denote by $\SLn^\mu$ the quotient of $\SLn$ under the move $\mu$.\\
  We shall use similar notation for classical and welded pure braids and links.
\end{nota}

\subsection{Relation between local moves}
\label{sec:Relations}

\begin{defi}
    Let $\M_1$ and $\M_2$ be two local moves.\\
 We say that $\M_2$ \emph{w--generates} $\M_1$ if $\M_1$ can be realized
  using $\M_2$ and welded Reidemeister moves. We denote it by 
  $M_2\stackrel{w}{\Rightarrow} M_1$. 
  If $M_2\stackrel{w}{\Rightarrow} M_1$ 
  and $M_1\stackrel{w}{\Rightarrow} M_2$, 
  then we say that $\M_1$ and $\M_2$ are \emph{w--equivalent}.\\
  If $\M_1$ and $\M_2$ are classical, then we say that $\M_2$ \emph{c--generates} $\M_1$ if $\M_1$ can be realized
  using $\M_2$ and classical Reidemeister moves. We denote it by 
  $M_2\stackrel{c}{\Rightarrow} M_1$. 
  If $M_2\stackrel{c}{\Rightarrow} M_1$ 
  and $M_1\stackrel{c}{\Rightarrow} M_2$, 
  then we say that $\M_1$ and $\M_2$ are \emph{c--equivalent}.
\end{defi}

\begin{remarque}
  The inclusion $\plong{\iota}{\SLDn}{\vSLDn}$ induces a well defined
  map $\func{\iota_*}{\SLn^{\M_c}}{\wSLn^{\M_w}}$ whenever
  $\M_w \stackrel{w}{\Rightarrow}\M_c$ with $\M_c$ a classical local move.
But the induced
  map is, in general, not injective. This shall be one of the main
  motivation for Definition \ref{def:Extension}.  
\end{remarque}

\begin{prop}\label{prop:F<=>UC}
  Move $\F$ is w--equivalent to the following under-commute move
\[
    \dessin{1.2cm}{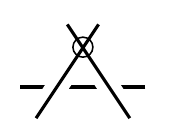} \stackrel{\UC}{\longleftrightarrow}
      \dessin{1.2cm}{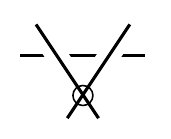}
    \hspace{.5cm}\leftrightsquigarrow\hspace{.5cm}
                     \dessin{1.5cm}{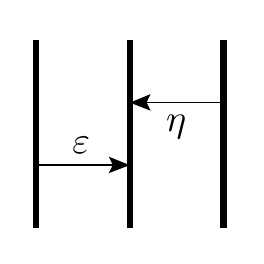} \stackrel{\UC}{\longleftrightarrow}
      \dessin{1.5cm}{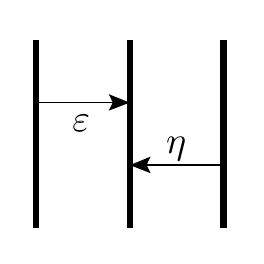}.
\]
\end{prop}
\begin{proof}
  From the Gauss diagram point of view, $\F$ can be realized as:
\[
\dessin{1.5cm}{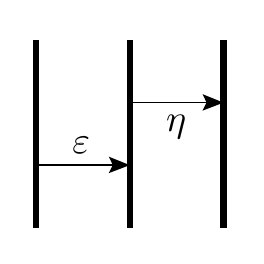}\ \xrightarrow[\RD]{}\ \dessin{1.5cm}{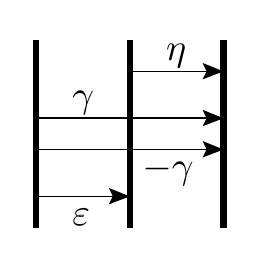}\ \xrightarrow[\OC]{}\ \dessin{1.5cm}{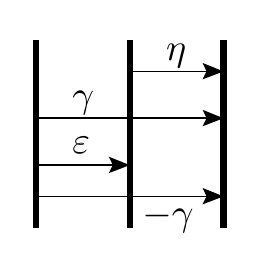}\ \xrightarrow[\RT]{}\ \dessin{1.5cm}{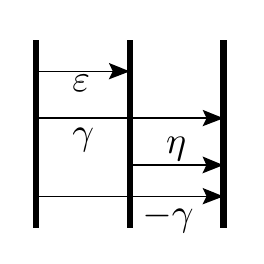}\ \xrightarrow[\UC]{}\ \dessin{1.5cm}{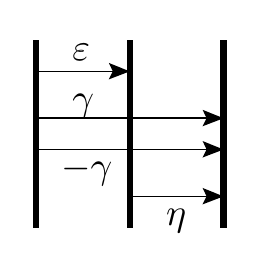}\ \xrightarrow[\RD]{}\ \dessin{1.5cm}{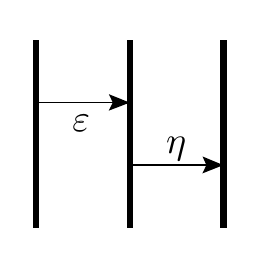}.
\]
Note that the restrictions on signs requested to perform the $\RT$
move can be fulfilled since we are free to choose the value of
$\gamma$ and free to choose the orientation of the piece of strand
that supports the tail of the $\varepsilon$--labeled arrow.

Conversely, $\UC$ can be similarly realized using $\F$.
\end{proof}
It follows, in particular, that $\wSLn^\F$ is actually the fused quotient studied in \cite{KL1,KaL} where both ``forbidden moves'' are allowed. From the Gauss diagram point of view, it has also the following consequence.
\begin{cor}\label{cor:Commutation}
  Up to $\F$, two arrow ends which are consecutive on a strand can be exchanged. Consequently, for Gauss diagram representatives of elements in $\wSLn^\F$, ends of arrows can be moved freely along a strand, so that the only relevant information is the strands it starts from and goes to. 
\end{cor}
\begin{prop}\label{prop:ARCC=>F}
  Both $\AR$ and $\CC$ w--generate $\F$.
\end{prop}
\begin{proof}
  From the Gauss diagram point of view, $\F$ can be realized as:
\[
\dessin{1.5cm}{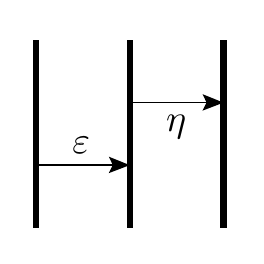}\ \xrightarrow[\AR\textrm{ or }\CC]{}\ \dessin{1.5cm}{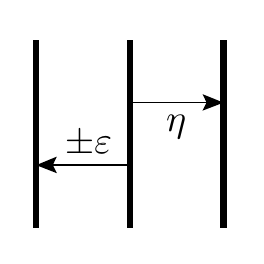}\ \xrightarrow[\OC]{}\ \dessin{1.5cm}{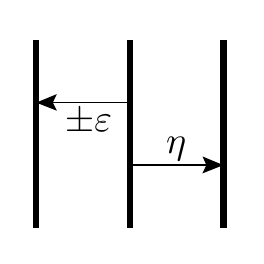}\ \xrightarrow[\AR\textrm{ or }\CC]{}\ \ \ \dessin{1.5cm}{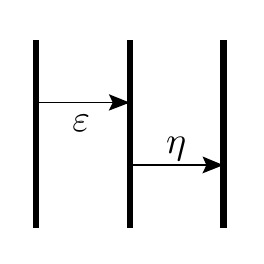}.
\]
\end{proof}

\begin{prop}\label{prop:wS=>SR}
  The move $\wSharp$ w--generates $\SR$.
\end{prop}
\begin{proof}
  From the virtual diagram point of view, $\SR$ can be realized as:
\[
\dessin{.9cm}{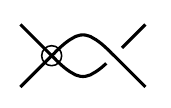}\ \xrightarrow[\textrm{wReid's}]{}\ \dessin{3.1cm}{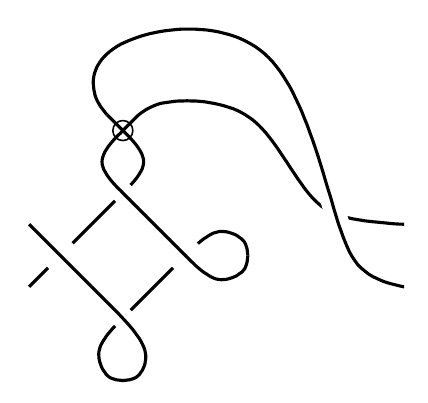}\ \xrightarrow[\wSharp]{}\ \dessin{3.1cm}{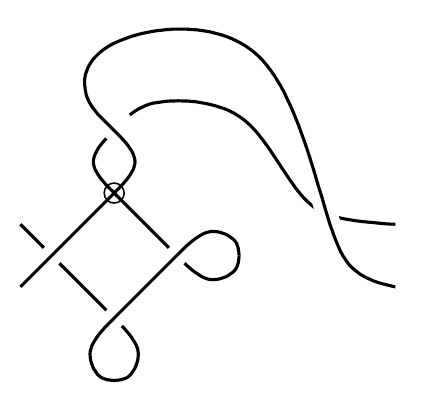}\ \xrightarrow[\textrm{wReid's}]{}\ \ \ \dessin{.9cm}{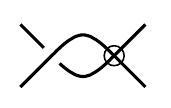}.
\]
\end{proof}

Now we state the first theorem of this section, which emphasizes the parallel between the classical and the welded realms.

\begin{theo}\label{th:LocMovesOrder}
The following diagram holds:
\[\vcenter{\hbox{$\xymatrix@!0@C=.6cm@R=1.1cm{
\CC \ar@{}[rr]|{\cprecc} \ar@{}[rd]|{\vprecc}&& \Sharp \ar@{}[rr]|{\cprecc} \ar@{}[rd]|{\vprecc}&& \Delta \ar@{}[rr]|{\cprecc} \ar@{}[rd]|{\vprecc}&& \SC \ar@{}[rd]|{\vprecc}&\\
&\V \ar@{}[rr]|{\wprecc}&& \wSharp \ar@{}[rr]|{\wprecc}&& \F \ar@{}[rr]|{\wprecc} &&\SV
}$}}.
\]
\end{theo}
\begin{proof}
  The upper line gathers classical known facts:
  \begin{itemize}
  \item $\CC \stackrel{c}{\Rightarrow} \Sharp$ is obvious;
  \item $\Delta \stackrel{c}{\Rightarrow} \SC$ is a classical fact, proved in \cite[Lemma 1.1]{MN};
  \item $\Sharp \stackrel{c}{\Rightarrow} \Delta$ is proved by combining \cite[Prop 1]{nakanishi2} and \cite[Lemma 2]{Aida};
        more precisely, Lemma 2 of \cite{Aida} shows (when ignoring orientations) that $\Sharp$ w--generates the following local move: 
        \[ \dessin{1.2cm}{Del_1} \ \longleftrightarrow \dessin{1.2cm}{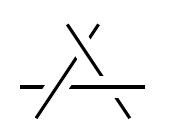}, \]
        while \cite[Prop 1 (5)]{nakanishi2} proves that the above local move w--generates $\Delta$. 
   \end{itemize}

  The statements $\SV \stackrel{w}{\Rightarrow} \SC$, $\V \stackrel{w}{\Rightarrow} \CC$ and $\V \stackrel{w}{\Rightarrow} \wSharp$ are direct consequences of the fact that virtualization $\V$ and $\SV$ moves allow to remove any arrow and reinsert it with reversed sign. 

  Since $\Sharp$ and $\wSharp$ differ by one application of $\SR$, the statement $\wSharp \stackrel{w}{\Rightarrow} \Sharp$ is a corollary of Proposition \ref{prop:wS=>SR}.

  From the Gauss diagram point of view, $\Delta$ modifies the relative positions of three arrows. 
  It follows hence from Corollary \ref{cor:Commutation} that $\F \stackrel{w}{\Rightarrow} \Delta$.
  For the same reason $\F \stackrel{w}{\Rightarrow} \SV$, since both ends of a self-arrow can be made adjacent and the self-arrow removed using $\RU$.

  To prove the last statement $\wSharp \stackrel{w}{\Rightarrow} \F$, we first note that $\wSharp$ w--generates the 
 \emph{4-move}:\footnote{it is still an open problem, known as the 4-move conjecture and first posed by Y.~Nakanishi \cite{nakanishi}, whether this move is an unknotting operation on classical knots.}
\[
 \dessin{.75cm}{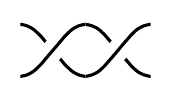}\  \stackrel{\Clasp}{\longleftrightarrow}
      \dessin{.75cm}{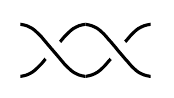}
    \hspace{.5cm}\leftrightsquigarrow\hspace{.5cm}
                     \dessin{1.5cm}{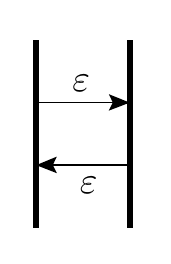}\  \stackrel{\Clasp}{\longleftrightarrow}
      \dessin{1.5cm}{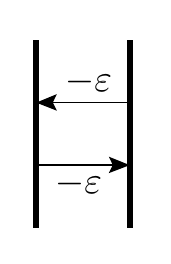}.
\]
Indeed, we have the following, which can be seen as a welded analogue of \cite[Fig 8]{Murakami}: 
\[
\dessin{.75cm}{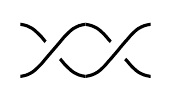}\ \xrightarrow[\textrm{wReid's}]{}\ \dessin{1.5cm}{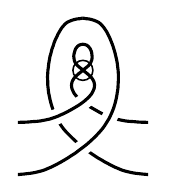}\ \xrightarrow[\wSharp]{}\ \dessin{1.5cm}{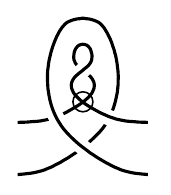}\ \xrightarrow[\textrm{wReid's}]{} \ \ \ 
\dessin{.75cm}{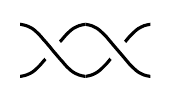}.
\]
The following sequence
\[
\dessin{2cm}{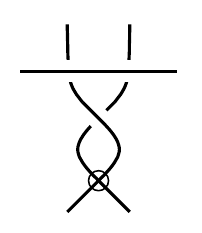}\ \xrightarrow[\textrm{wReid's}]{}\ \dessin{2cm}{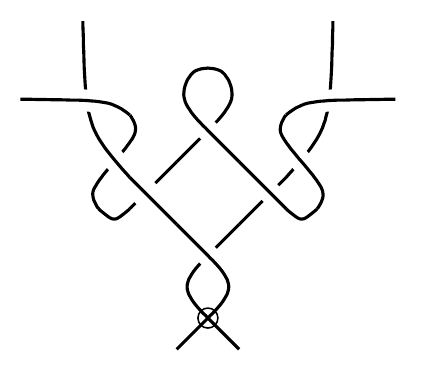}\ \xrightarrow[\wSharp]{}\ \dessin{2cm}{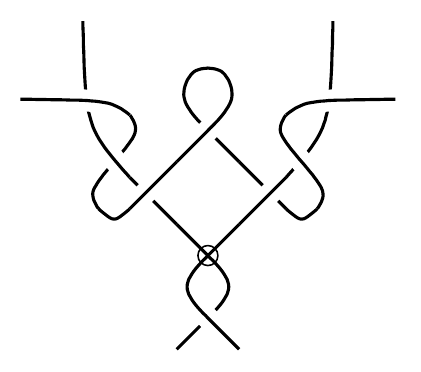}\ \xrightarrow[\Clasp]{} \ \ \ 
\dessin{2cm}{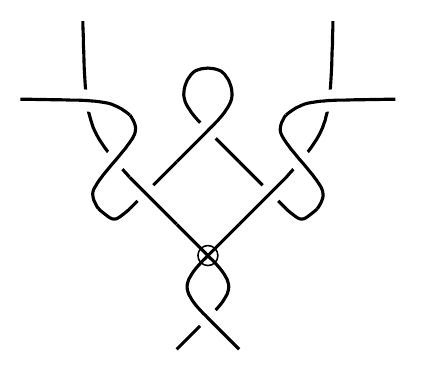}\ \xrightarrow[\textrm{wReid's}]{}\ \dessin{2cm}{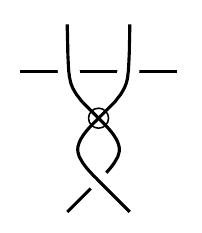}.
\]
proves hence that $\wSharp$ w--generates the following move:
\[
\dessin{1.2cm}{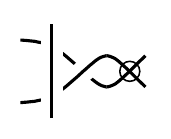} \ \stackrel{\M}{\longleftrightarrow}
        \dessin{1.2cm}{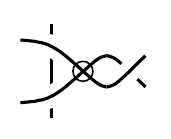}
    \hspace{.5cm}\leftrightsquigarrow\hspace{.5cm}
                     \dessin{1.5cm}{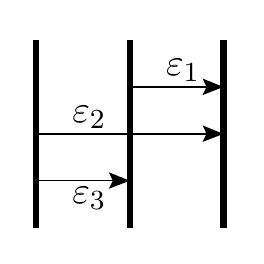}\ \stackrel{\M}{\longleftrightarrow}
      \dessin{1.5cm}{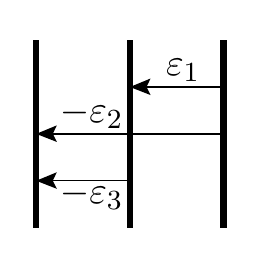}.
\]
The following sequence, together with Proposition \ref{prop:F<=>UC}, then concludes  the proof:
\[
\dessin{1.5cm}{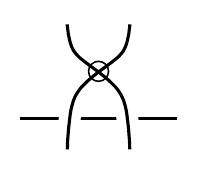}\ \xrightarrow[\RD]{}\ \dessin{1.8cm}{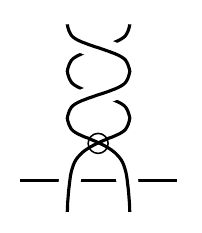}\ \xrightarrow[\M]{}\ \dessin{1.8cm}{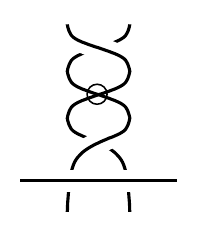}\ \xrightarrow[\textrm{wReid's}]{}\ \ \ \dessin{1.8cm}{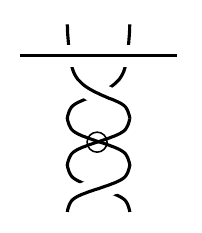}\ \xrightarrow[\M]{}\ \dessin{1.8cm}{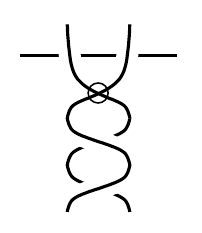}\ \xrightarrow[\RD]{}\ \dessin{1.5cm}{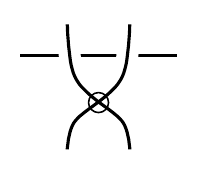}.
\]
\end{proof}

Note that if a statement
$\M_2 \stackrel{c}{\Rightarrow} \M_1$, for $\M_1$ and $\M_2$ some classical local moves, is proved by realizing locally $\M_1$ using $\M_2$,
then it automatically follows that $\M_2 \stackrel{w}{\Rightarrow} \M_1$. For instance, it
follows directly from the proofs in the classical case that
$\CC \stackrel{w}{\Rightarrow} \Sharp \stackrel{w}{\Rightarrow} \Delta$. On the contrary, the proof that $\Delta \stackrel{c}{\Rightarrow} \SC$ is
not local and promoting it to the welded realm requires some
attention.  

\begin{theo}\label{th:Del=>SC}
The following relations hold: 
  $\CC \stackrel{w}{\Rightarrow} \Sharp \stackrel{w}{\Rightarrow} \Delta \stackrel{w}{\Rightarrow} \SC$.
\end{theo}

\begin{proof}
  As already noted, only the relation $\Delta \stackrel{w}{\Rightarrow} \SC$ needs to be proved. We
  shall adopt the Gauss diagram point of view. Let $a$ be a self-arrow
  on which one wants to realize a crossing change. We shall proceed by induction on the \emph{width} of $a$, which is defined as the number of heads located on the
portion of strand in-between the two endpoints of $a$.

If $a$ has width zero, then there is no head between the endpoints of $a$, and the crossing change can be realized as:
  \[
\dessin{2.5cm}{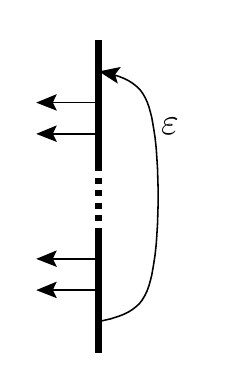}\ \xrightarrow[\OC\textrm{'s}]{}\ 
\dessin{2.5cm}{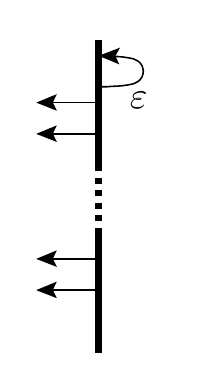}\ \xrightarrow[\RU]{}\ 
\dessin{2.5cm}{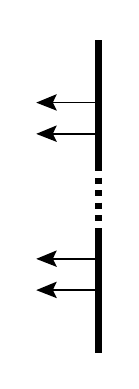}\ \xrightarrow[\RU]{}\ 
\dessin{2.5cm}{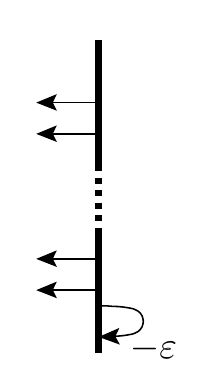}\ \xrightarrow[\OC\textrm{'s}]{}\ 
\dessin{2.5cm}{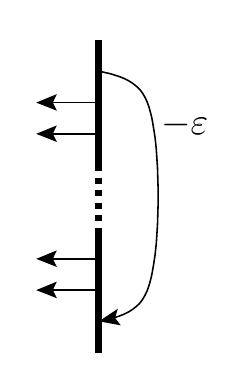}.
\]
Now, we assume that $a$ has width $d\in\N^*$, 
and that the statement is proven for self-arrows having width smaller than $d$. 
We call an \emph{interior} arrow any self-arrow which has both endpoints
located in the portion of strand between the endpoints of
$a$. There are hence two cases: 
\begin{description}
\item[There is an interior arrow $b$] then we proceed in three steps.
  \begin{description}
  \item[Step 1] Remove $b$ by pushing its tail next to its head, as follows.  
    Tails can be crossed using $\OC$. Heads from non interior arrows
    can be crossed using the sequence:
\[
\dessin{2.5cm}{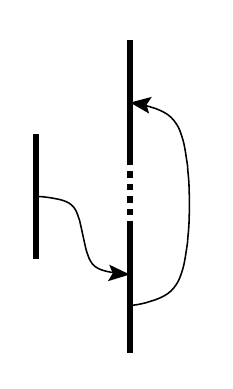}\ \xrightarrow[\RD]{}\ 
\dessin{2.5cm}{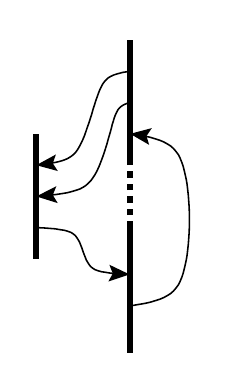}\ \xrightarrow[\Delta]{}\ 
\dessin{2.5cm}{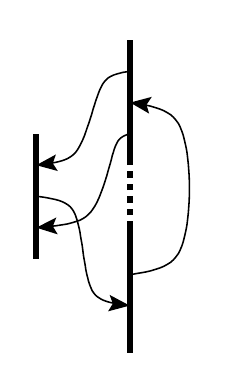}.
\]
The restrictions on signs
    requested to perform the $\Delta$ move can be fulfilled since we
    are free to choose the signs of the arrows created with the $\RD$
    move and free to choose the orientation of the piece of strand
    that supports the tail of the non interior arrow.
  Heads from interior arrows can be crossed by using the induction
  hypothesis, which allows to turn them into tails using self-crossing changes;
  an interior arrow has indeed a strictly smaller width than $a$. The arrow $b$ can now be removed using $\RU$.
\item[Step 2]   
  Since none of the operations of Step 1 has increased the number of
  head between its endpoints, $a$ has now width $d-1$ and the induction hypothesis can be used to perform a
  self-crossing change on it.
\item[Step 3] The arrow $b$ can be placed back by performing Step 1 backwards. 
  \end{description}
\item[There is no interior arrow] then we also proceed in three steps.
  \begin{description}
  \item[Step 1] Push the tail of $a$ towards its head until it has crossed one head. 
    In doing so, the tail of $a$ first crosses a number of tails (of non interior arrows), 
    and we request that these are not crossed using $\OC$ 
    but using the sequence:
\[
\dessin{2.7cm}{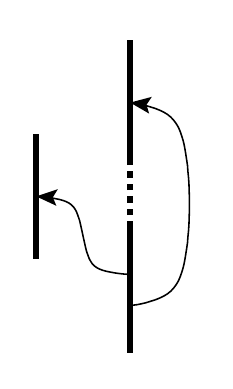}\ \xrightarrow[\RD]{}\ 
\dessin{2.7cm}{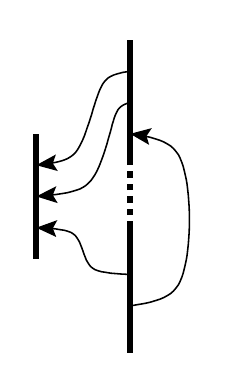}\ \xrightarrow[\RT]{}\ 
\dessin{2.7cm}{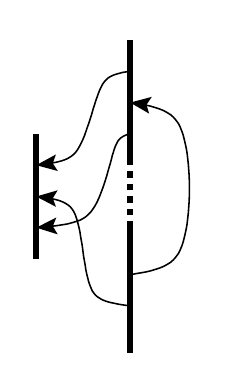}.
\]
The restrictions on signs
    requested to perform the $\RT$ move can be fulfilled since we
    are free to choose the signs of the arrows created with the $\RD$
    move and free to choose the orientation of the piece of strand
    that supports the head of the non interior arrow. Finally, the first head met by the tail of $a$ is crossed using the sequence given in Step 1 of the previous case.
  \item[Step 2] Since none of the operations of Step 1 has increased the number of
  head between its endpoints, $a$ has now width $d-1$ and the induction hypothesis can be used to perform a  self-crossing change on it.
\item[Step 3] The tail of $a$ can now be pushed back to its initial
  position by performing Step 1 backwards. It is indeed illustrated in
  Figure \ref{fig:RestrictionsIdem} that the $\Delta$ and $\RT$ moves performed in Step 1, and the corresponding 
    move performed in this final step,
    have sign restrictions which are simultaneously satisfied\footnote{this can also be trivially
    checked from the virtual diagram point of view}. 
    Some random orientations have been chosen for the strands in
    Figure \ref{fig:RestrictionsIdem}, but changing it would merely
    add a sign on both sides.
  \end{description}
  \begin{figure}
\[
\hspace{-.5cm}
    \begin{array}{ccc}
      \dessin{2.7cm}{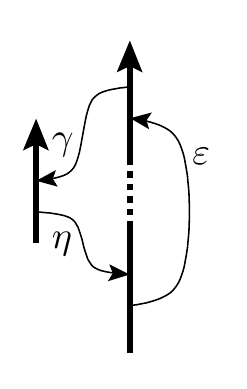}&\crashto&\dessin{1.8cm}{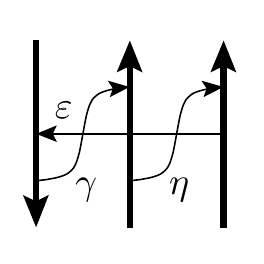}\\[-.7cm]
       & & \crashdown\\[0.2cm]
       & & \varepsilon=-\eta=\gamma\\[0.2cm]
       & & \crashup\\[-.6cm]
      \dessin{2.7cm}{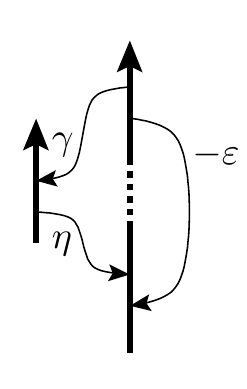}&\crashto&\dessin{1.8cm}{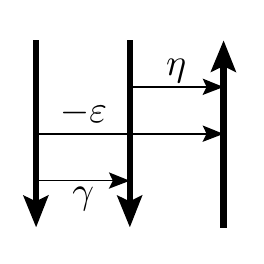}
    \end{array}
    \hspace{2cm}
    \begin{array}{ccc}
      \dessin{2.7cm}{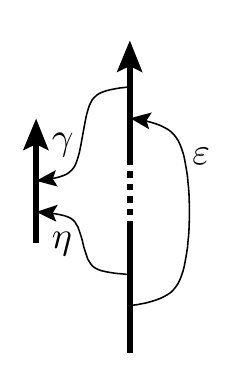}&\crashto&\dessin{1.8cm}{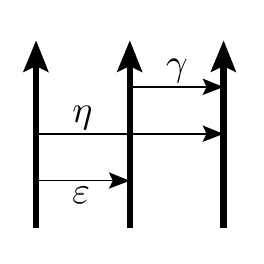}\\[-.7cm]
       & & \crashdown\\[0.2cm]
       & & \varepsilon=\eta=\gamma\\[0.2cm]
       & & \crashup\\[-.6cm]
      \dessin{2.7cm}{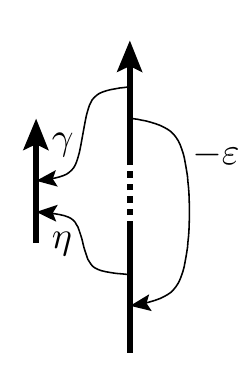}&\crashto&\dessin{1.8cm}{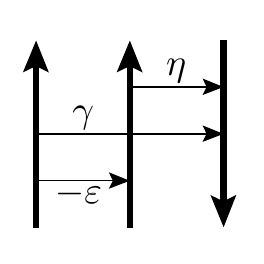}
    \end{array}
\]
    \caption{Correspondence between sign restrictions}\label{fig:RestrictionsIdem}
  \end{figure}
\end{description}
\end{proof}

As observed in \cite[Lemma 4.4]{vaskho2}, $\SC$ is an unknotting operation on $\wSL_1$. 
Indeed, by $\SC$ and $\OC$ moves, any Gauss diagram of a long knot can be turned into a diagram where all arrows have adjacent endpoints, 
which is clearly trivial by R1.
It follows that Theorem \ref{th:Del=>SC} and Proposition \ref{prop:ARCC=>F} have the
following corollary.
\begin{cor}\label{cor:Unknotting}
  Moves $\Delta$, $\Sharp$, $\F$, $\AR$ and $\wSharp$ are all unknotting
  operations on welded long knots, and hence on welded knots.
\end{cor}

\begin{remarque}
  The  statement of Corollary \ref{cor:Unknotting} on welded knots has been also recently proved in a
  different way by S. Satoh \cite{Satoh2}.
\end{remarque}

A consequence of Corollary \ref{cor:Unknotting} is that, on one
strand ---that is for welded long knots--- all the local moves considered
in Theorem \ref{th:LocMovesOrder} are w--equivalent.
In the next section, we shall provide classification results which point out that, except in a few cases (see {\it e.g.} Corollary \ref{cor:SVF}), this is no longer true on more strands. However, $\wSLn^\Delta$ is not considered
there, and since it provides another example of w--equivalence on two strands, we address it now.
\begin{prop}\label{prop:SC=>Del}
  On two strands, $\Delta$ and $\SC$ are w--equivalent, i.e. we have  $\wSL_2^\Delta=\wSL_2^\SC$.  
\end{prop}
\begin{proof}
  It has been proved in Theorem \ref{th:Del=>SC} that $\Delta \stackrel{w}{\Rightarrow}\SC$. 
  Conversely, any $\Delta$ move involves at least two pieces
  of strand which belong to the same strand. By performing a
  self-crossing change on the corresponding crossing before and after,
  the $\Delta$ move can then be replaced by an $\RT$.
\end{proof}
\begin{remarque}
  On three strands and more, Milnor invariant $\mu^w_{i_1i_2i_3}$,
defined in \cite[Sec 5.2]{vaskho} ---but which are also described,
in terms of Gauss diagram formula,  as
$\left\langle\dessin{.7cm}{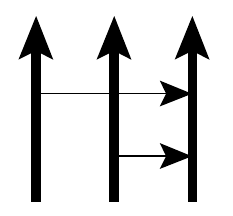}+\dessin{.7cm}{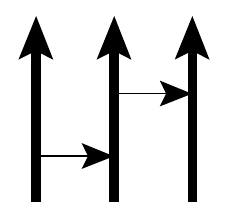}-\dessin{.7cm}{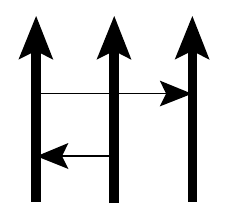},
  - \right\rangle$, using notation from \cite[Sec 3.2]{vaskho2}--- detects any 
$\Delta$ move, but is invariant under $\SC$.

Note also that, even on two strands, $\wSL_2^\F$ is a proper
quotient of $\wSL_2^\Delta$ since the elements
\[
S: \dessin{2.5cm}{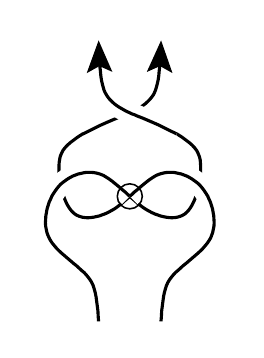} \hspace{.2cm}\leftrightsquigarrow\hspace{.2cm}
\dessin{1.8cm}{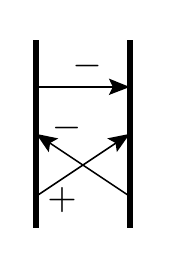}
\hspace{2cm}\textrm{and}\hspace{2cm} S': 
\dessin{2.5cm}{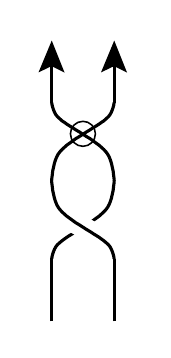} \hspace{.2cm}\leftrightsquigarrow\hspace{.2cm}
\dessin{1.8cm}{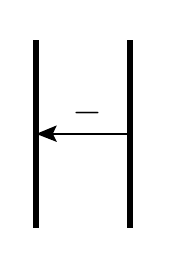}.
\]
are equal in the former but not in the latter.
Indeed, the invariant $Q_2$, defined in the proof of \cite[Lemma 4.10]{vaskho2} as the invariant
for welded string links up to $\SC$ given by $\left\langle\dessin{.7cm}{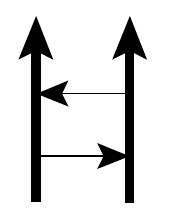}-\dessin{.7cm}{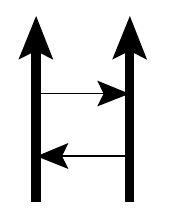}+\dessin{.7cm}{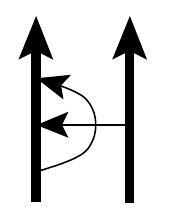}-\dessin{.7cm}{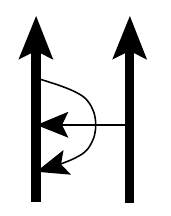}+\dessin{.7cm}{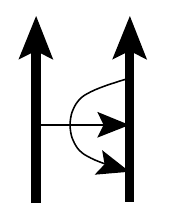}-\dessin{.7cm}{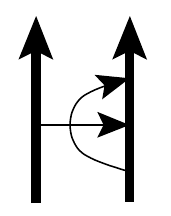},
  - \right\rangle$
is also invariant under $\Delta$: 
 if a $\Delta$ move involves three distinct
strands, only one of the three arrows affected by this move can be involved in the
computation of $Q_2$, so that the value of $Q_2$ is the same before and after the
move; if a $\Delta$ move involves only one or two distinct strands,
then it was noticed in the proof of Proposition \ref{prop:SC=>Del} that
it can be replaced by two $\SC$ and one $\RT$ moves. 
The invariant $Q_2$ is hence well defined on $\wSL_2^\Delta$, and it is directly computed that $Q_2(S)=-1$ whereas $Q_2(S')=0$.
\end{remarque}


\section{Classifying invariants}
\label{sec:Classification}
 

In this section, we classify string link diagrams
modulo the main local moves studied above
and as a corollary, we discuss how some classical
local moves can be extended to the welded case.
A global description of the results is given in Figure \ref{fig:BigDiag}.

\begin{figure}
  \[
\xymatrix@!0@C=3cm@R=.7cm{
\AutC^0(\RFn) &\Z^{\frac{n(n-1)}{2}}& & \Z_2^{n-1} && 0\\
&&&&&\\
&&&&&\\
\SLn^\SC \ar@{->>}[r] \ar[uuu]_{\varphi_\HL}^-\vcong \ar@{^(->}[ddddd]&\SLn^\Delta
\ar@{->>}[rr] \ar[uuu]^-\vcong_-{\big(\lk_{ij}\big)_{1\leq i<j\leq n}}
\ar@{^(->}[dddr] \ar@{^(->}[ddddd]&& \SLn^\Sharp
\ar@{->>}[rr]\ar[uuu]^-\vcong_-{\big(\lk^{\m}_{i\ast}\big)_{1\leq i\leq n-1}}\ar@{^(->}[ddddd]&&
\SLn^\CC 
\ar[uuu]^-\vcong \ar@{^(->}[ddddd]\ar@{_(->}[ldddd]\\
&&&&&\\
&&&&&\\
&&\wSLn^\AR
\ar[dddddd]_(.666)\vcong^(.666){\big(\vlk_{ij}+\vlk_{ji}\big)_{1\leq
  i<j\leq n}}|(.212)\hole |(.333)\hole&&&\\
&&&&\wSLn^\CC
\ar@{->>}[rd]\ar[ddddd]_(.6)\vcong^(.6){\big(\vlk_{ij}-\vlk_{ji}\big)_{1\leq
  i<j\leq n}}|!{[dl];[dr]}\hole&\\
\wSLn^\SV
\ar@{->>}[r]\ar[dddd]^-{\varphi^w_\HL}_-\vcong&\wSLn^\F
\ar@/^.2cm/@{->>}[ruu]\ar@{->>}[rr]\ar[dddd]_-\vcong^{\big(\vlk_{ij}\big)_{1\leq
  i\neq j\leq n}}\ar@/^.3cm/@{->>}[rrru]|(.67)\hole&&\wSLn^{\wSharp}
  \ar@{->>}[rr]
\ar[dddd]_-\vcong^{\substack{\big(\vlk^{\m}_{ij}+\vlk^{\m}_{ji}\big)_{1\leq i<j\leq n}\\\oplus\\\big(\vlk^{\m}_{i\ast}\big)_{1\leq i\leq n-1}}}&&\wSLn^{\V}
\ar[dddd]_-\vcong\\
&&&&&\\
&&&&&\\
&&&&&\\
\AutC(\RFn)&\Z^{n(n-1)}&\Z^{\frac{n(n-1)}{2}} &\Z_2^{\frac{(n+2)(n-1)}{2}}&\Z^{\frac{n(n-1)}{2}}&0
}
\]
  \caption{Summary of the classification and extension results}
  \label{fig:BigDiag}
\end{figure}
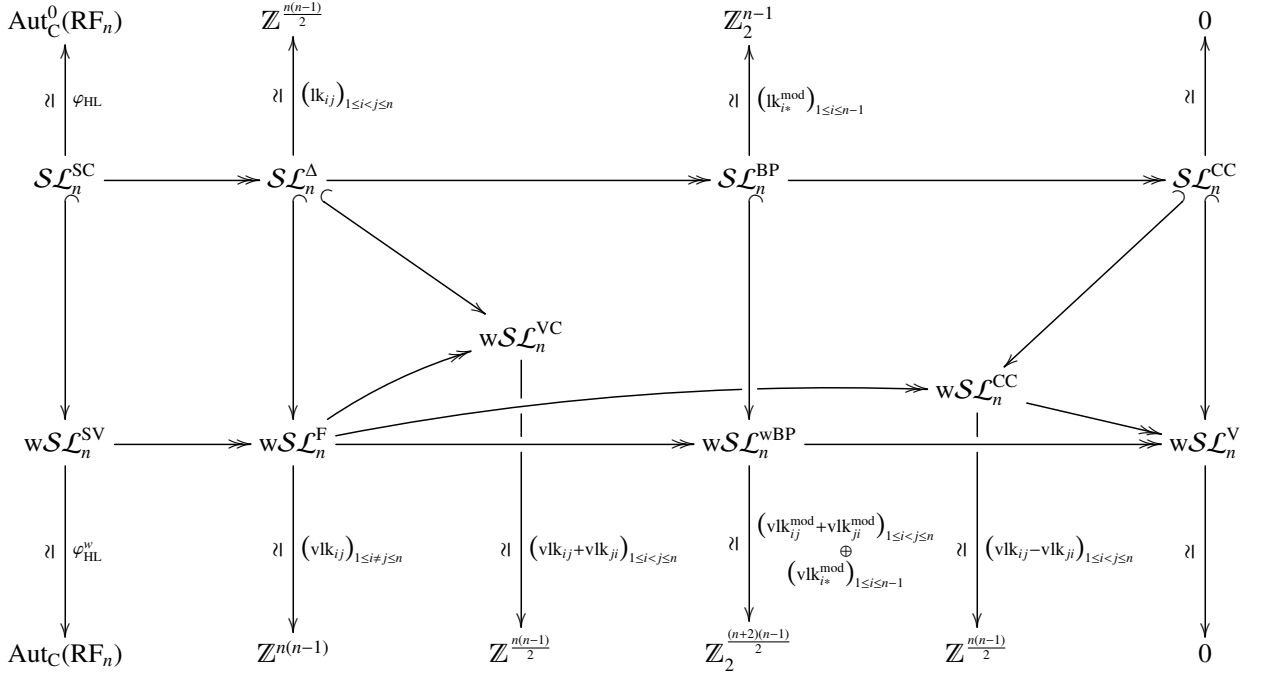

\subsection{Classifications}

\begin{defi}
  Let $\mu$ be a local move and $\func{\phi}{\vSLDn}{A}$, for some monoid $A$, be a morphism of monoids.\\
We say that $\phi$ \emph{w--classifies} $\mu$ if $\phi$ is invariant under $\mu$ and the welded Reidemeister moves, and that the induced map $\func{\phi_*}{\wSLn^\mu}{A}$ is an isomorphism.\\
If $\mu$ is classical, we say that $\phi$ \emph{c--classifies} $\mu$ if $\phi$ is invariant under $\mu$ and the classical Reidemeister moves, and that the induced map $\func{\phi_*}{\SLn^\mu}{A}$ is an isomorphism.
\end{defi}

Now, we define a few welded invariants which shall classify the moves introduced in the previous section.

\begin{defi}
  For every $i\neq j\in\{1,\ldots,n\}$, we define the \emph{virtual linking number} $\func{\vlk_{ij}}{\vSLD}{\Z}$ as the map which counts, with signs, the crossings where the $i^\textrm{th}$ component passes over the $j^\textrm{th}$ component. 
From the Gauss diagram point of view, it simply counts the signs of all the arrows going from the $i^\textrm{th}$ to the $j^\textrm{th}$ strands.
\end{defi}

The following lemma is considered as folklore. The first part is clear from the Gauss diagram point of view  and the second part can be proved using the virtual diagram point of view.

\begin{lemme}
  For every $i\neq j\in\{1,\ldots,n\}$, $\vlk_{ij}$ is invariant under
  welded Reidemeister moves and, if $D\in\SLDn$ is a classical
  diagram, then $\vlk_{ij}(D)=\vlk_{ji}(D)$.
\end{lemme}
\begin{nota}
  For every $i\neq j\in\{1,\ldots,n\}$, we set
  \begin{itemize}
  \item the \emph{linking number} $\func{\lk_{ij}}{\SLDn}{\Z}$ as the restriction to classical diagrams of either (and
    equivalently) $\vlk_{ij}$,
    $\vlk_{ji}$ or $\frac{1}{2}\big(\vlk_{ij}+\vlk_{ji}\big)$;
  \item  $\vlk_{i*}:=\psum_{\substack{1\leq k\leq n\\k\neq i}}\vlk_{ik}$
    and $\lk_{i*}:=\psum_{\substack{1\leq
        k\leq n\\k\neq i}}\lk_{ik}$;
  \item $\vlk_{ij}^\m$, $\vlk_{i*}^\m$, $\lk_{ij}^\m$ and
    $\lk_{i*}^\m$ as the modulo 2 reduction of $\vlk_{ij}$, $\vlk_{i*}$, $\lk_{ij}$ and $\lk_{i*}$, respectively.
  \end{itemize}
\end{nota}
Linking and virtual linking numbers can be similarly defined for classical or welded pure braids and links.

\medskip

In \cite{HL}, Habegger and Lin defined a map
$\func{\varphi_\HL}{\SLDn}{\AutC^0(\RFn)}$ which was extended into a map
$\func{\varphi_\HL^w}{\vSLDn}{\AutC(\RFn)}$ in \cite{vaskho}, in
the sense that $\varphi_w\circ\iota=\varphi$.
Here, $\RFn$ denotes the largest quotient of the free group over $x_1,\ldots,x_n$ such that each $x_i$ commutes with all its conjugates, $\AutC(\RFn)$ is the group of automorphisms of $\RFn$ mapping each $x_i$ to a conjugate of itself, and $\AutC^0(\RFn)$ is the subgroup of such  automorphisms fixing $x_1\cdots x_n$. 
\medskip 

The following classification results are known.
\begin{prop}\label{prop:whitney'strique}$\ $\\
In the classical case: 
 \begin{itemize}
  \item {\cite[Thm 1.1]{MN}\footnote{in the given references, the statements are for links rather than string links, but as discussed in Section \ref{sec:Links}, up to $\Delta$ or $\Sharp$, these notions are the same\label{footnote1}}}\, 
  The map $\func{\big(\lk_{ij}\big)_{1\leq i<j\leq
      n}}{\SLDn}{\Z^{\frac{n(n-1)}{2}}}$ c--classifies $\Delta$. 
  \item {\cite[Thm 11.6.7]{Kawauchi} and \cite[Thm A.2]{MN}\refmark{footnote1}}\,
  The map $\func{\big(\lk_{i*}^\m\big)_{1\leq i\leq
      n-1}}{\SLDn}{\Z_2^{n-1}}$ c--classifies $\Sharp$. 
  \item {\cite[Thm 1.7]{HL}}\,
  The map $\func{\varphi_\HL}{\SLDn}{\AutC^0(\RFn)}$ c--classifies $\SC$.
\end{itemize}
In the welded case: 
 \begin{itemize}
  \item {\cite[Thm 2.34]{vaskho}}\, The map $\func{\varphi_\HL^w}{\vSLDn}{\AutC(\RFn)}$ w--classifies $\SV$.
\end{itemize}
\end{prop}

We now provide new classification results. To this end, we use the Gauss diagram point of view and define, for every $k\in\N$, $\e\in\{\pm1\}$ and $i\neq j\in\{1,\ldots,n\}$, $G^{\e k}_{i,j}$ to be the Gauss diagram which has only $k$ horizontal $\e$--labeled arrows
from strand $i$ to $j$.

\begin{prop}\label{prop:ClassificationF}
  The map $\func{\big(\vlk_{ij}\big)_{1\neq i<j\leq
      n}}{\vSLDn}{\Z^{n(n-1)}}$ w--classifies $\F$.
\end{prop}
\begin{proof}
It essentially follows from Corollary \ref{cor:Commutation} that virtual linking numbers are invariant under $\F$.

It can be noted that $G^{\e k}_{i,j}$ satisfies $\vlk_{i,j}(G^{\e
  k}_{i,j})=\e k$ and $\vlk_{p,q}(G^{\e
  k}_{i,j})=0$ for $(p,q)\neq(i,j)$.
By stacking such Gauss
diagrams in lexicographical order of $i\neq j\in\{1,\ldots,n\}$, we
obtain normal forms realizing any configuration of the virtual linking
numbers.

Now, given a Gauss diagram, all self-arrow can be removed using $\F$ and welded Reidemeister moves since $\F \stackrel{w}{\Rightarrow} \SC$, and then, using Corollary \ref{cor:Commutation}, arrow ends can be reorganized in order to obtain one of
the above normal forms.
\end{proof}

\begin{prop}\label{prop:ClassificationAR}
  The map $\func{\big(\vlk_{ij}+\vlk_{ji}\big)_{1\leq i<j\leq
      n}}{\vSLDn}{\Z^{\frac{n(n-1)}{2}}}$ w--classifies $\AR$.
\end{prop}
\begin{proof}
  Performing $\AR$ on a self-arrow does not affect any $\vlk_{ij}+\vlk_{ji}$. Performing it on an arrow between strands $i$ and $j$ adds $\pm1$ to $\vlk_{ij}$ while it adds $\mp1$ to $\vlk_{ji}$; the sum $\vlk_{ij}+\vlk_{ji}$ hence remains invariant.

  Surjectivity of the induced map is achieved by considering the same normal forms than in the proof of Proposition \ref{prop:ClassificationF}, but restricted to 
$G^{\e k}_{i,j}$ with $i<j$.
  
  Given a Gauss diagram, all $\vlk_{ij}$ with $i>j$ can be made to vanish by creating sufficiently many arrows from the $i^\textrm{th}$ to the $j^\textrm{th}$ strand, using $\RD$ moves, and performing $\AR$ on them. Then, by Proposition \ref{prop:ClassificationF}, 
  there is a sequence of $\F$ and welded Reidemeister moves to one of the above normal forms, and this concludes the proof since $\AR \stackrel{w}{\Rightarrow} \F$.
\end{proof}

\begin{prop}\label{prop:ClassificationCC}
  The map $\func{\big(\vlk_{ij}-\vlk_{ji}\big)_{1\leq i<j\leq
      n}}{\vSLDn}{\Z^{\frac{n(n-1)}{2}}}$ w--classifies $\CC$.
\end{prop}
\begin{proof}
  The proof is totally similar to that of Proposition \ref{prop:ClassificationAR}.
\end{proof}

\begin{prop}
    The map $\func{\big(\vlk^\m_{ij}+\vlk^\m_{ji}\big)_{1\leq i<j\leq
      n}\oplus \big(\vlk^\m_{i*}\big)_{1\leq i\leq
      n-1}}{\vSLDn}{\Z_2^{\frac{n(n-1)}{2}}\oplus\Z_2^{n-1}=\Z_2^{\frac{(n+2)(n-1)}{2}}}$ w--classifies $\wSharp$.
\end{prop}
\begin{proof}
  A move $\wSharp$ is a combination of $\CC$ and $\AR$. As such, it modifies $\vlk_{ij}+\vlk_{ji}$ by a multiple of 2. Moreover, a move $\wSharp$ changes the number of non-self arrows with the tail on a given strand by 0 or 2: this is clear if none of the four involved pieces of strands are connected, and it can be case-by-case checked in the other situations. As a consequence, $\vlk_{i*}$ is also modified by a multiple of 2. In conclusion, the invariant is indeed invariant under $\wSharp$.

 Surjectivity of the induced map is achieved by stacking (in any order) elements of the form $G^{+1}_{i,j}$ as follows: fix an element in $\Z_2^{\frac{(n+2)(n-1)}{2}}$, seen as some $\big(\vlk^\m_{ij}+\vlk^\m_{ji}\big)_{1\leq i<j\leq n}\oplus \big(\vlk^\m_{i*}\big)_{1\leq i\leq n-1}$; start from the Gauss diagram with no arrow and, for each $i<j\in\{1,\ldots,n\}$, add one $G^{+1}_{i,j}$ whenever $\vlk^\m_{ij}+\vlk^\m_{ji}\equiv 1$; next, for each $i\in\{1,\ldots,n-1\}$, add $G^{+1}_{i,n}$ and $G^{+1}_{n,i}$ if $\vlk_{i*}(G)\neq \vlk_{i*}$ and do nothing otherwise. The resulting Gauss diagram is the requested preimage.

Injectivity of the induced map is proved by induction on $n$. For
$n=1$, the result follows from Corollary \ref{cor:Unknotting}. Now
assume that $n>1$ and that the result is true on $n-1$ strands. It was
shown in \cite[Thm 4.12]{vaskho} that every welded string link is related to a welded braid by welded Reidemeister and $\SV$ moves. Since every welded braid has an inverse, 
and since $\wSharp \stackrel{w}{\Rightarrow} \SV$, it follows that $\wSLn^{\wSharp}$ is a group and it is thus sufficient 
to prove that the kernel of the induced map is trivial. Consider hence a Gauss diagram $G$ which is in the kernel. In the following, we shall apply some welded Reidemeister and $\wSharp$ moves on $G$, but by abuse of notation we shall keep denoting it by $G$. 
First, we can modify $G$ so that each $\vlk_{ij}(G)$ is either 0 or 1. Indeed, this is easily achieved using the $\SR$ move, and $\wSharp$ w-generates $\SR$ by Proposition \ref{prop:wS=>SR}.  
Next, we show how to reduce all $\vlk_{1i}(G)$ and $\vlk_{i1}(G)$ to 0. 
If $\vlk_{1i}(G)$ is 1 for some $i\neq 1$, then there is $j\neq1,i$ such that   
$\vlk_{1j}(G)=1$, for otherwise $\vlk^\m_{1*}(G)$  would be 1; 
moreover, we have that $\vlk_{i1}(G)$ is also 1, for otherwise $\vlk_{1i}^\m(G)+\vlk_{i1}^\m(G)$ would be 1, 
and likewise we have $\vlk_{j1}(G)=1$.   
Then perform locally the following sequence anywhere on $G$ (there, the indices at the bottom correspond to the labels of the strands to which the different pieces belong):
\[ 
\dessin{3.2cm}{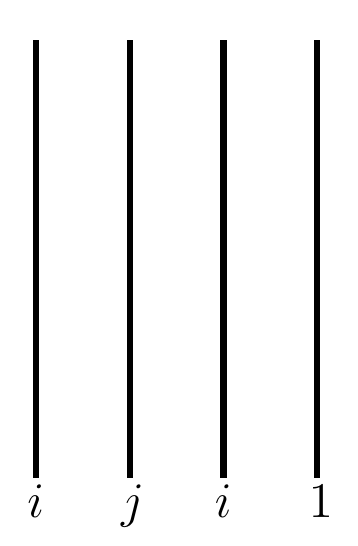}\ \xrightarrow[\RD\textrm{'s}]{}\ \dessin{3.2cm}{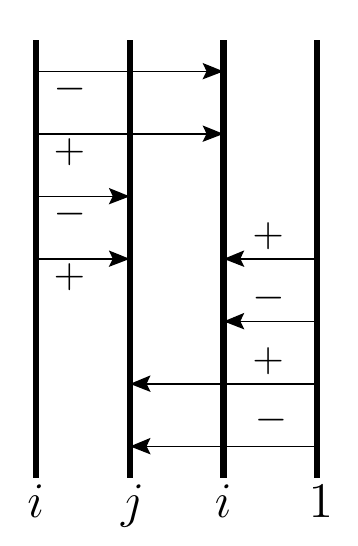}\ \xrightarrow[\OC\textrm{'s}]{}\ \dessin{3.2cm}{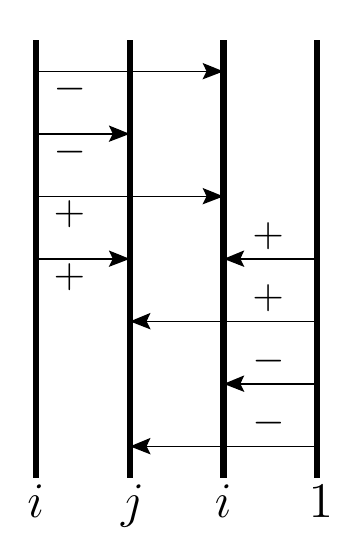}\ \xrightarrow[\wSharp]{}\ \dessin{3.2cm}{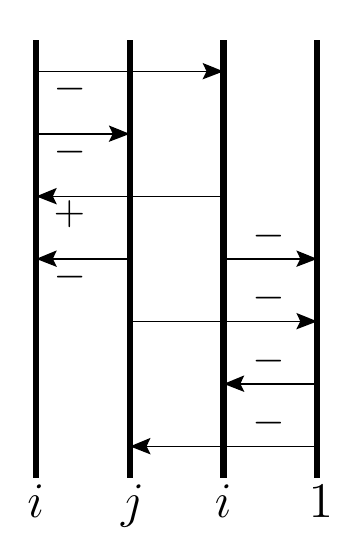}.
\]
As a result, we have that $\vlk_{1i}(G)=\vlk_{i1}(G)=\vlk_{1j}(G)=\vlk_{j1}(G)=0$. 
Repeat this operation until all $\vlk_{1i}(G)$ and $\vlk_{i1}(G)$ are 0, as desired. Using the normal form given in the proof of Proposition \ref{prop:ClassificationF}, there is a sequence of welded Reidemeister and $\F$ moves transforming $G$ into a Gauss diagram with no arrow touching the first strand. 
Since $\wSharp \stackrel{w}{\Rightarrow} \F$, this sequence can be traded for a sequence of welded Reidemeister and $\wSharp$ moves. By forgetting the first strand, we obtain a welded string link on $n-1$ strands which is in the kernel of the induced map. By induction hypothesis, $G$ is hence trivial in $\wSLn^{\wSharp}$.
\end{proof}

Classification results may be used to prove w--equivalence  
between local moves.
\begin{cor}\label{cor:SVF}
  On two strands, $\SV$ and $\F$ are w--equivalent, but for $n\geq3$, $\wSLn^\F$ is a proper quotient of $\wSLn^\SV$.
\end{cor}
\begin{proof}
  The local move $\F$ is w-classified by virtual linking numbers, i.e. $\wSLn^\F$ is isomorphic to $\Z^{n(n-1)}$. 
On the other hand, $\wSLn^\SV$ is isomorphic to  $\AutC(RF_n)$. 
For $n=2$, it is easily seen that any element of $\Aut_C(\RF_2)$ can be written as $\varphi_{n_1,n_2}$, 
for some integers $n_1, n_2 \in \N$, where $\varphi_{n_1,n_2} (x_1)=x_2^{n_1} x_1 x_2^{-n_1}$
and $\varphi_{n_1,n_2}(x_2)= x_1^{n_2} x_2 x_1^{-n_2}$, and that $\varphi_{n_1,n_2} \varphi_{n_3,n_4}
=\varphi_{n_1+n_3,n_2+n_4}$. 
This implies shows that $\AutC(RF_2)$ is isomorphic $\Z^2$, while $\AutC(RF_n)$ is not abelian for $n>2$. 
\end{proof}
\begin{remarque}
  The fact that, on two strands, $\SV \stackrel{w}{\Rightarrow} \F$ can also be seen as a corollary of \cite[Prop 4.11]{vaskho}.
\end{remarque}

\subsection{Welded extensions}

\begin{defi}\label{def:Extension}
  Let $\M_c$ and $\M_w$ be two local moves such
  that $\M_c$ is classical and $\M_w \stackrel{w}{\Rightarrow} \M_c$. We say that $\M_w$ extends $\M_c$ if the map
  $\func{\iota_*}{\SLn^{\M_c}}{\wSLn^{\M_w}}$, induced by the
  inclusion $\plong{\iota}{\SLDn}{\vSLDn}$, is injective.
\end{defi}

This definition is motivated by the following direct consequence.

\begin{lemme}\label{lem:Extension}
  Let $\M_c$ and $\M_w$ be two local moves, such that $\M_c$ is
  classical and $\M_w$ extends $\M_c$. If two classical diagrams are
  connected by a sequence of $\M_w$ and welded Reidemeister moves,
  then they are connected by a sequence of $\M_c$ and classical
  Reidemeister moves. 
\end{lemme}

Now, we can use the classification results of the previous section to
obtain some extension results. In each case, it is sufficient to check that the target of the c--classifying map can be identified with a subset of the target of the w--classifying map and that, with this identification, the c--classifying map is actually the composition of the w--classifying map with the injection $\plong{\iota}{\SLDn}{\vSLDn}$.

\begin{prop}[{\cite[Thm 4.3]{vaskho2}}]
  The move $\SV$ extends $\SC$.
\end{prop}

\begin{prop}\label{prop:FextendsD}
  Both $\F$ and $\AR$ extend $\Delta$.
\end{prop}
This proposition, as well as Proposition \ref{prop:VCC->CC}, illustrates how a given classical local move may be extended in several different ways.
\begin{proof}
  To prove that $\F$ extends $\Delta$, we use Propositions \ref{prop:whitney'strique} and \ref{prop:ClassificationF}, and identify $\Z^{\frac{n(n-1)}{2}}$ with the subset of $\Z^{n(n-1)}$ made of elements such that the $ij$ and $ji$--summands are equal for every $i\neq j\in\{1,\ldots,n\}$. The linking number $\lk_{ij}$ is then seen as simultaneously equal to $\vlk_{ij}$ and $\vlk_{ji}$.

  Similarly, to prove that $\AR$ extends $\Delta$ using Propositions \ref{prop:whitney'strique} and \ref{prop:ClassificationAR}, we identify $\Z^{\frac{n(n-1)}{2}}$ with the subset  $2\Z^{\frac{n(n-1)}{2}}$ of even-valued elements in $\Z^{\frac{n(n-1)}{2}}$. The linking number $\lk_{ij}$ should then rather be interpreted as $\frac{1}{2}\big(\vlk_{ij}+\vlk_{ji}\big)$.
\end{proof}

\begin{prop}\label{prop:VCC->CC}
  Both $\V$ and $\CC$ extend $\CC$.
\end{prop}
\begin{proof}
  In both situation, there is no ambiguity on how $0$ is identified as a subset of $0$ or of $\Z^{\frac{n(n-1)}{2}}$.
  In the latter case, the result follows from the fact that $\vlk_{ij}=\vlk_{ji}$ on classical diagrams, so that $\vlk_{ij}-\vlk_{ji}$ vanishes.
\end{proof}

\begin{prop}\label{prop:wsharp_sharp}
  The local move $\wSharp$ extends $\Sharp$.
\end{prop}
\begin{proof}
To prove the statement, $\Z_2^{n-1}$ should be identified with the $n-1$ last $(\vlk^\m_{i*})$--summands of $\Z_2^{\frac{(n+2)(n-1)}{2}}$, the other being identically equal to 0.  
\end{proof}


\section{Braids and links}
\label{sec:Others}
 
We now investigate how the results given so far for string links can be transported to the more familiar context of braids and (possibly unordered) links. 
In particular, we prove the results stated in the introduction. 
\subsection{Welded links}
\label{sec:Links}

Most of the results stated for classical and welded string links in this paper
extend to the case of welded links. 
In particular, it can be noted that none
of the proofs given in section \ref{sec:Relations} uses the fact that
we are dealing with intervals rather than with circles. 
It follows that all generation and equivalence
results stated there hold the same for classical and welded links. 

\subsubsection{From string links to links}
There is a natural way to associate a  welded link to a welded string
link, incarnated by the closure map $\func{\Cl}{\wSLn}{\wCLn}$ which is defined,
using the Gauss diagram point of view, by identifying
  pairwise the endpoints of each strand while keeping the order on the
  resulting circles. This map restricts to a well defined map
  $\func{\Cl}{\SLn}{\CLn}$.
  Cutting circles into intervals 
produces preimages, showing that these
  maps are surjective. However, this procedure does not
  provide a well defined inverse for the closure map since
  it strongly depends on an arbitrary choice of cutting points; and indeed, except for $\func{\Cl}{\SL_1}{\CL_1}$, the closure maps
  are not injective.
A noteworthy consequence of Corollary \ref{cor:Commutation} is that, up to
$\F$ moves, and consequently up to most local moves considered in this paper, the procedure does provide a well-defined inverse, proving that the quotiented notions of welded string links and  welded links coincide. 

\begin{prop}\label{prop:wSLF=wCLF}
  The closure map induces one-to-one
  correspondences between $\wSLn^\mu$ and $\wCLn^\mu$ for $\mu=\F$, $\CC$, $\AR$ or $\wSharp$.
\end{prop}

\begin{proof}
  It is sufficient to prove that, up to welded Reidemeister and the considered local moves, the opening procedure described above does not depend on the chosen cutting points. The resulting map $\func{\Op}{\wCLn^\mu}{\wSLn^\mu}$ would then clearly satisfy $\Op\circ\Cl=\Id_{\wSLn^\mu}$ and $\Cl\circ\Op=\Id_{\wCLn^\mu}$. In order to prove such an independence, it is sufficient, on the  link side seen from the Gauss diagram point of view, to show that a cutting point can cross an arrow endpoint. On the string link side, it corresponds to moving the endpoint from one extremity of the strand to the other. Because of  Corollary \ref{cor:Commutation}, this can be done freely using $\F$ or any other local move which w--generates $\F$.
\end{proof}

Combined with the extension results given in the previous section, Proposition \ref{prop:wSLF=wCLF} induces similar statements for classical objects.

\begin{prop}\label{prop:SLF=CLF}
  The closure map induces one-to-one
  correspondences between $\SLn^\mu$ and $\CLn^\mu$ for $\mu=\Delta$ or $\Sharp$. 
\end{prop}
\begin{proof}
  Let $D_1$ and $D_2$ be two classical string link diagrams which have, through the closure map, the same image in $\CLn^\Delta$, 
  and hence in $\wCLn^\F$ since $\F \stackrel{w}{\Rightarrow} \Delta$. According to Proposition \ref{prop:wSLF=wCLF}, they are connected by a sequence of welded Reidemeister and $\F$ moves. But, by Proposition \ref{prop:FextendsD} and Lemma \ref{lem:Extension}, they are hence connected by classical Reidemeister and $\Delta$ moves and thus represent the same element of $\SLn^\Delta$.

The statement for $\Sharp$ is proved similarly.
\end{proof}

This statement can be independently proved using the fact that linking numbers simultaneously
 classify  links \cite[Thm 1.1]{MN} and string links \cite[Thm 4.6]{jbjktr} up to $\Delta$.
\medskip 

It follows from Propositions \ref{prop:wSLF=wCLF} and \ref{prop:SLF=CLF} that most classification and extension results given in Section \ref{sec:Classification} hold the same for classical and welded  links, as stated in Theorems \ref{th:ColoredLinksClassification} 
and \ref{th:ColoredLinksExtension}.

\subsubsection{Unordered links}
There is an obvious action of the symmetric group $\SS_n$ on classical and welded  links which simply permutes the order on the components. Classical and welded unordered links are the natural quotients under this action. For each classification given in Theorem \ref{th:ColoredLinksClassification}, 
an action of $\SS_n$ can be defined on the target space so that it results in a classification for the considered local move for unordered (welded) links. For instance, the target space of the classification of welded  links up to $\F$ is $\Z^{n(n-1)}\cong\Z^{\left\{\bigstrut(i,j)\ |\ 1\leq i\neq j\leq n\right\}}$. For every $\sigma\in\SS_n$ and $(a_{ij})_{1\leq i\neq j\leq n}\in\Z^{n(n-1)}$, we can set $\sigma.(a_{ij})_{1\leq i\neq j\leq n}:=(a_{\sigma(i)\sigma(j)})_{1\leq i\neq j\leq n}$; unordered welded links up to $\F$ are then classified by the symmetrized virtual linking numbers $\func{\vlk_{ij}}{\vSLDn}{\fract{\Z^{n(n-1)}}/{\SS_n}}$.
More generally, we have the following. 
\begin{prop}\label{prop:LinksClassification}
  \begin{itemize}
  \item[]
  \item Unordered links up to $\Delta$ are classified by the symmetrized linking numbers.
  \item Unordered links up to $\Sharp$ are classified by the symmetrized $\lk_{i*}^\m$'s.
  \item Unordered welded links up to $\F$ are classified by the symmetrized virtual linking numbers.
  \item Unordered welded links up to $\AR$ are classified by the symmetrized $(\vlk_{ij}+\vlk_{ji})$'s.
  \item Unordered welded links up to $\CC$ are classified by the symmetrized $(\vlk_{ij}-\vlk_{ji})$'s.
  \item Unordered welded links up to $\wSharp$ are classified by the symmetrized $\vlk_{i*}^\m$'s and $(\vlk_{ij}^\m+\vlk_{ji}^\m)$'s.
  \end{itemize}
\end{prop}

The target spaces of the above classifications are, in general, not particularly nicely described. A notable exception is the case of classical links up to $\Sharp$ which reduces to the number of $i$ such that $\lk_{i*}^\m$ is 1; see \cite[Thm 11.6.7]{Kawauchi} and \cite[Thm A.2]{MN}.

These classifications can, in turn, be used to show extension results, that is, we have a strict analogue of Theorem \ref{th:ColoredLinksExtension} 
in the unordered case. 

\subsection{Welded braids}

It is well known that pure braids embed in string links (it was actually noticed by Artin; see for instance Chapter 6, Proposition 1.1
of \cite{MK} for a complete proof)
and, similarly
welded pure braid embed in welded string links (\cite[Rk 3.7]{vaskho}).
However this embedding is not, in general, preserved when we add other local moves.
 For instance, elements of $\wPn^\F$, called \emph{unrestricted pure virtual braids} in
\cite{BBD}, do not embed in $\wSLn^\F$.
Indeed, in \cite[Thm 2.7]{BBD}, $\wPn^\F$ is proved to be non abelian while $\wSLn^\F$ is; it follows that there are adjacent endpoints of arrows which do
  not commute. This difference lies in the fact that, in the proof of Proposition
  \ref{prop:F<=>UC}, we allowed for the introduction of non horizontal arrows 
  and, in particular, self-arrows. 
  Indeed, in the sequence proving Proposition \ref{prop:F<=>UC}, the first and
  third strands may be part of the same component if the two initial arrows
  have endpoints on the same two strands. In this case, the $\RD$ move
  creates two self-arrows. 

In another direction, one may wonder if a given local move is strong enough to make surjective the natural map from (welded) pure braids to (welded) string links.
In the classical case, the following is known. 
\begin{prop}\cite[p.399]{HL}
The natural embedding of $\Pn$ in $\SLn$ induces
an isomorphism between $\Pn^\SC$ and $\SLn^\SC$.
\end{prop}
Note here that, since $\SC$ requires the presence of a self-crossing, which a pure braid cannot contain, it should be understood in the above statement that $\Pn$  
stands for its embedding in $\SLn$.  
The same holds for the $\SV$ move in the welded settings, addressed in the next result. 
\begin{prop} \label{prop:braids} The natural embedding of $\wPn$ in $\wSL$ induces:
\begin{itemize}
\item  an isomorphism between $\wPn^\SV$  
and $\wSLn^\SV$;
\item an isomorphism between $\wPn^\CC$ and $\wSLn^\CC$.
\end{itemize}
\end{prop}
\begin{proof}
The first statement was proved in \cite[Thm 3.11]{vaskho}. The second follows from the fact that 
the quotient of $\wPn$ by $\CC$ is isomorphic to the flat welded pure braid group, introduced in \cite[Sec 5.2]{BBD},
which is isomorphic to $\Z^{\frac{n(n-1)}{2}}$; moreover a straightforward verification shows that this isomorphism is
realized  by the  $(\vlk_{ij}-\vlk_{ji})$'s. The statement follows then from the classification of $\wSLn^{CC}$ given in Proposition \ref{prop:ClassificationCC}. 
\end{proof}

On the other hand,  it is known that $\wPn^\SC$ does not coincide
  with  $\wSLn^\SC$ (\cite[Lemma 4.8]{vaskho2}). It seems interesting to analyze the case of other local moves, but for this one should consider their oriented versions,
in order to have  the right analogue for welded pure braids.


\bibliography{Fused}{}
\bibliographystyle{plain}

\end{document}